\documentclass[preprint,12pt]{elsarticle}

\usepackage{amsmath,amsfonts,amssymb,amsthm}
\usepackage{enumitem}
\usepackage{tikz}
\usetikzlibrary{angles,quotes,calc,arrows.meta,patterns,positioning,shapes.geometric}
\usepackage[hidelinks]{hyperref}
\usepackage{cleveref}

\newtheorem{theorem}{Theorem}[section]
\newtheorem{lemma}[theorem]{Lemma}
\newtheorem{proposition}[theorem]{Proposition}
\newtheorem{corollary}[theorem]{Corollary}
\theoremstyle{remark}
\newtheorem{remark}[theorem]{Remark}
\newtheorem{problem}[theorem]{Problem}

\begin{document}

\begin{frontmatter}

\title{Quantitative and Uniform $L^2$ Non-Localization on Integrable Polygons}

\author[hcmus,vnuhcm]{Binh T. Nguyen\corref{cor1}}
\ead{ngtbinh@hcmus.edu.vn}

\cortext[cor1]{Corresponding author.}
\affiliation[hcmus]{organization={Faculty of Mathematics and Computer Science, University of Science},
            city={Ho Chi Minh City},
            country={Vietnam}}
\affiliation[vnuhcm]{organization={Vietnam National University Ho Chi Minh City},
            city={Ho Chi Minh City},
            country={Vietnam}}

\begin{abstract}
We study uniform lower bounds for the local $L^2$ mass of
eigenfunctions on planar integrable polygons, with particular emphasis
on spectral degeneracy and quantitative dependence on the observation
set. Given a bounded domain $\Omega\subset\mathbb{R}^2$ and a
measurable set $V\subset\Omega$ of positive measure, we consider
\[
    C_2(V;\Omega)
    :=
    \inf_{\lambda\in\sigma(-\Delta_\Omega)}
    \inf_{0\neq u\in E_\lambda(\Omega)}
    \frac{\|u\|_{L^2(V)}}{\|u\|_{L^2(\Omega)}}.
\]
The second infimum ranges over the complete eigenspace and therefore
allows arbitrary cancellations among modes associated with a multiple
eigenvalue.

We prove that $C_2(V;\Omega)>0$ for the four integrable polygonal
classes in the trigonometric spectral classification: rectangles,
isosceles right triangles, equilateral triangles, and
hemi-equilateral triangles. Thus every fixed positive-measure
observation set captures a uniformly positive fraction of the
$L^2$ mass throughout the spectrum, independently of spectral
multiplicity.
For rectangles, we obtain a quantitative refinement. If the reflected
extension of $V$ to the associated rectangular torus has finite
perimeter and
\[
    \alpha=\frac{|V|}{|\Omega|},
\]
we derive an explicit sufficient spectral threshold
$\lambda_*(\Omega,V)$ such that every eigenfunction with
$\lambda\geq\lambda_*(\Omega,V)$ satisfies
\[
    \frac{\|u\|_{L^2(V)}}{\|u\|_{L^2(\Omega)}}
    \geq
    \left[
        \frac{\alpha}{2}
        \left(
            1-\frac{\sin(\pi\alpha)}{\pi\alpha}
        \right)
    \right]^{1/2}.
\]
Finally, we establish two robustness properties. The stationary
non-localization estimate extends to Dirichlet Schr\"odinger operators
with bounded real-valued potentials on the rectangular branch.
Moreover, an explicit spectral-defect estimate extends observability
from exact eigenfunctions to sufficiently accurate quasimodes and to
arbitrary vectors in sufficiently narrow spectral clusters.
\end{abstract}


\end{frontmatter}

\section{Introduction}
\label{sec:introduction}

Let $\Omega\subset\mathbb{R}^2$ be a bounded planar domain, and let
$-\Delta_\Omega$ denote the Dirichlet Laplacian on $L^2(\Omega)$.
We consider the following eigenvalue problem
\begin{equation}
    -\Delta_\Omega u=\lambda u,
    \qquad
    u\in D(-\Delta_\Omega)\setminus\{0\}.
    \label{eq:intro-eigenproblem}
\end{equation}
Here and throughout, $-\Delta_\Omega$ is understood as the
nonnegative self-adjoint realization of $-\Delta$ with Dirichlet
boundary conditions; its precise variational definition is recalled
below.

The spatial distribution of Laplacian eigenfunctions is a central
problem in spectral theory and mathematical physics and depends
sensitively on the geometry of the underlying domain. Of particular
interest is the localization phenomenon of eigenfunctions, in which a significant
fraction of the $L^2$ mass can concentrate in a restricted portion
of the domain. Such behavior may arise from trapping, geometric
bottlenecks, boundary effects, or concentration near distinguished
classical trajectories; see, Grebenkov et al. \cite{GrebenkovNguyen2013Review} and the references therein.

For a measurable set $V\subset\Omega$ with $|V|>0$, we consider the
uniform $L^2$ norm ratio
\begin{equation}
    C_2(V;\Omega)
    :=
    \inf_{\lambda\in\sigma(-\Delta_\Omega)}
    \inf_{\substack{u\in E_\lambda(\Omega)\setminus\{0\}}}
    \frac{\|u\|_{L^2(V)}}
         {\|u\|_{L^2(\Omega)}},
    \qquad
    E_\lambda(\Omega)
    :=
    \ker(-\Delta_\Omega-\lambda I).
    \label{eq:intro-C2}
\end{equation}
The condition
\[
    C_2(V;\Omega)>0
\]
means that every Dirichlet eigenfunction has a uniformly positive
$L^2$ norm on $V$ relative to its global $L^2$ norm, with a lower
bound independent of the eigenvalue and of the particular vector in
the corresponding eigenspace. Consequently, no sequence of
$L^2$-normalized eigenfunctions can asymptotically avoid a fixed
measurable subset of positive measure.

For separated eigenfunctions on rectangular domains, related lower
bounds can be obtained directly from explicit product formulas.
Nguyen et al. ~\cite{GrebenkovNguyen2013} considered the more
general quantity
\begin{equation}
    C_p(V)
    :=
    \inf_j
    \frac{\|u_j\|_{L^p(V)}}
         {\|u_j\|_{L^p(\Omega)}},
\end{equation}
and established its positivity for every nonempty open subset in
rectangle-like domains with simple spectrum.
They also pointed out that spectral degeneracies introduce
an additional difficulty: an eigenfunction associated with a multiple
eigenvalue may be an arbitrary linear combination of basis elements
in the corresponding eigenspace, and estimates for individual
separated modes do not automatically yield estimates uniform over
all such combinations. The positivity problem was therefore left
unresolved for general rectangle-like domains in the presence of
spectral degeneracies. This difficulty already occurs for the square.

From the viewpoint of the localization of Laplacian eigenfunctions, this leads naturally
to the following uniform non-localization problem: under Dirichlet
boundary conditions, does every measurable set
$V\subset\Omega$ of positive measure capture a uniformly positive
fraction of the $L^2$ norm of every Laplacian eigenfunction, uniformly
over the spectrum and over all vectors in each eigenspace? In the
notation above, the question is whether
\[
    C_2(V;\Omega)>0
\]
holds without any simplicity assumption on the Dirichlet spectrum.
For integrable polygonal domains, resolving this question amounts to
passing from mode-by-mode information to an eigenspace-uniform
statement that remains valid in the presence of spectral
multiplicities.

Our first aim is to formulate this result within the $L^2$
non-localization framework, with uniformity over each eigenspace.
The Schr\"odinger observability estimates used below are known.
If $u$ belongs to an eigenspace of the Dirichlet Laplacian, then its
Schr\"odinger evolution differs from $u$ only by a scalar phase.
Applying the observability inequality to such a $u$ therefore gives
an $L^2$ lower bound on the observation set that holds for every
vector in the eigenspace, including arbitrary linear combinations
when the eigenvalue has nontrivial multiplicity.

More precisely, suppose that for some observation time $\tau>0$,
\begin{equation}
    \|f\|_{L^2(\Omega)}^2
    \leq
    C_{V,\tau}
    \int_0^\tau\int_V
    \left|e^{it\Delta_\Omega}f(x)\right|^2
    \,dx\,dt
    \label{eq:intro-observability}
\end{equation}
holds for every $f\in L^2(\Omega)$. If
$u\in E_\lambda(\Omega)$, then
\begin{equation}
    e^{it\Delta_\Omega}u
    =
    e^{-it\lambda}u,
    \label{eq:intro-eigen-phase}
\end{equation}
and therefore,
\begin{equation}
    \|u\|_{L^2(V)}
    \geq
    (\tau C_{V,\tau})^{-1/2}
    \|u\|_{L^2(\Omega)}.
    \label{eq:intro-obs-to-C2}
\end{equation}
The observability constant is independent of both the initial data
and the spectral parameter. Consequently, after restriction to an
eigenspace, the resulting lower bound is uniform in the eigenvalue
and over all nonzero vectors in that eigenspace.

We use this principle with the observability results available for
integrable polygons. For rectangles, odd reflection followed by
periodic extension transfers the Dirichlet problem to a rectangular
flat torus. The required Schr\"odinger observability estimate then
follows from Burq and Zworski
\cite[Theorem~1.3 and Remark~1]{BurqZworski2019}.
For equilateral triangles, we use the observability theorem of
Alphonse and Lafontaine
\cite[Theorem~1.2]{AlphonseLafontaine2025}, based on the triangular
tiling and a reduction to a rational twisted torus. The other two
integrable polygonal classes follow by reflection: an isosceles right
triangle reflects across its hypotenuse to a square, while a
hemi-equilateral triangle reflects across its altitude to an
equilateral triangle.

Combining these existing observability inputs with the symmetry
reductions described above yields a spectrum-wide non-localization
statement for the four integrable polygonal classes considered here.
More precisely, if $\Omega$ is a rectangle, an isosceles right
triangle, an equilateral triangle, or a hemi-equilateral triangle,
then for every measurable set $V\subset\Omega$ with $|V|>0$,
\begin{equation}
{
    C_2(V;\Omega)>0.
    }
    \label{eq:intro-main-result}
\end{equation}
Equivalently, there exists $c_{\Omega,V}>0$ such that
\begin{equation}
    \|u\|_{L^2(V)}
    \geq
    c_{\Omega,V}
    \|u\|_{L^2(\Omega)}
    \label{eq:intro-main-bound}
\end{equation}
for every nonzero Dirichlet Laplacian eigenfunction $u$, with
$c_{\Omega,V}$ independent of the eigenvalue, its multiplicity, and
the choice of $u$ within the corresponding eigenspace.

Beyond this eigenspace-uniform result for the free Dirichlet
Laplacian, we establish two robustness extensions. First, on
rectangles and isosceles right triangles, we consider Dirichlet
Schr\"odinger operators
\[
    H_{\Omega,q}
    =
    -\Delta_\Omega+q,
    \qquad
    q\in L^\infty(\Omega;\mathbb{R}),
\]
and prove uniform non-localization on every nonempty open observation
set. The argument combines odd reflection of the eigenfunction, even
reflection of the potential, and the stationary estimate of Bourgain,
Burq, and Zworski
\cite[Theorem~1]{BourgainBurqZworski2013}
on rectangular flat tori.

Second, we establish a quantitative stability estimate with respect
to the spectral defect. For a self-adjoint operator $A$ satisfying a
Schr\"odinger observability estimate, we prove
\begin{equation}
    \|u\|_{L^2(V)}
    \geq
    \frac{1}{\sqrt{\tau C_{V,\tau}}}
    \|u\|_{L^2(\Omega)}
    -
    \frac{\tau}{\sqrt{3}}
    \|(A-\lambda)u\|_{L^2(\Omega)}.
    \label{eq:intro-quasimode-bound}
\end{equation}
Exact eigenfunctions correspond to zero spectral defect. More
generally, the inequality~\eqref{eq:intro-quasimode-bound} yields uniform
non-localization for sufficiently accurate quasimodes and for
arbitrary vectors in sufficiently narrow spectral clusters. Thus, the
eigenspace-uniform result extends from linear combinations within a
single eigenvalue to linear combinations across sufficiently narrow
spectral windows.

The main results and contributions of the paper can be summarized as follows:
\begin{enumerate}[label=(\roman*)]

\item
\emph{Quantitative non-localization on rectangles.}
For a rectangle $R=(0,a)\times(0,b)$ and a measurable observation set
$V\subset R$ whose reflected extension has finite perimeter, we derive
a quantitative lower bound for the local $L^2$ mass over complete
Dirichlet eigenspaces.  We obtain an explicit sufficient threshold
$\lambda_*(R,V)$ determined by the reflected perimeter and the area
fraction
\[
    \alpha=\frac{|V|}{|R|}.
\]
Above this threshold, the geometric dependence disappears from the
mass floor:
\[
    \frac{\|u\|_{L^2(V)}}{\|u\|_{L^2(R)}}
    \geq
    \left[
        \frac{\alpha}{2}
        \left(
            1-
            \frac{\sin(\pi\alpha)}{\pi\alpha}
        \right)
    \right]^{1/2}.
\]
A finite-bandwidth Tur\'an--Nazarov argument controls the complementary
low-frequency regime and yields the explicit full-spectrum estimate
$C_2(V;R)\geq F_R(V)>0$.

\item
\emph{Eigenspace-level non-localization in the presence of spectral
degeneracy.}
For the complete class of integrable polygons in the trigonometric
spectral classification---rectangles, isosceles right triangles,
equilateral triangles, and hemi-equilateral triangles---we prove that
for every measurable set $V\subset\Omega$ with $|V|>0$,
\[
    C_2(V;\Omega)
    :=
    \inf_{\lambda\in\sigma(-\Delta_\Omega)}
    \inf_{\substack{u\in E_\lambda(\Omega)\setminus\{0\}}}
    \frac{\|u\|_{L^2(V)}}{\|u\|_{L^2(\Omega)}}
    >0.
\]
The estimate is uniform over the full spectrum and, crucially, over
the complete eigenspace associated with each eigenvalue. Hence no
simplicity or uniform multiplicity assumption is required, and
arbitrary cancellations inside degenerate eigenspaces are controlled.
This resolves the $L^2$ non-localization problem for the integrable
polygonal subclass at the eigenspace level.

\item
\emph{Robustness under bounded potentials on the rectangular branch.}
For rectangles and isosceles right triangles, we extend the
eigenspace-uniform non-localization estimate to Dirichlet
Schr\"odinger operators
\[
    -\Delta_\Omega+q,
    \qquad
    q\in L^\infty(\Omega;\mathbb R).
\]
The argument couples odd reflection of the eigenfunction with even
reflection of the potential and reduces the problem to a stationary
estimate on a rectangular flat torus. The resulting lower bound is
uniform in the eigenvalue, its multiplicity, and the choice of vector
inside the corresponding eigenspace.

\item
\emph{Quantitative stability beyond exact eigenspaces.}
For a self-adjoint operator $A$ whose Schr\"odinger flow satisfies an
observability inequality, we derive the explicit defect estimate
\[
    \|u\|_{L^2(V)}
    \geq
    \frac{1}{\sqrt{\tau C_{V,\tau}}}
    \|u\|_{L^2(\Omega)}
    -
    \frac{\tau}{\sqrt{3}}
    \|(A-\lambda)u\|_{L^2(\Omega)}.
\]
This gives a quantitative stability principle with respect to spectral
defect: exact eigenfunctions correspond to zero defect, while
sufficiently accurate quasimodes inherit a positive observation bound.
Through the spectral theorem, the same estimate yields uniform
non-localization for arbitrary vectors in sufficiently narrow spectral
clusters, thereby extending the result from degenerate eigenspaces to
superpositions of nearby eigenvalues.

\end{enumerate}

\paragraph{Notation and conventions.}
Throughout the paper, $\Omega\subset\mathbb{R}^2$ denotes a bounded
polygonal domain. We write $-\Delta_\Omega$ for the nonnegative
self-adjoint Dirichlet Laplacian on $L^2(\Omega)$ associated with the
closed quadratic form
\[
    q_\Omega(u,v)
    :=
    \int_\Omega
    \nabla u\cdot\nabla\overline{v}\,dx,
    \qquad
    D(q_\Omega)=H_0^1(\Omega).
\]
Thus, $u\in D(-\Delta_\Omega)$ and
$-\Delta_\Omega u=f$ if and only if
\[
    \int_\Omega
    \nabla u\cdot\nabla\overline{\varphi}\,dx
    =
    \int_\Omega
    f\,\overline{\varphi}\,dx
    \qquad
    \text{for every }\varphi\in H_0^1(\Omega).
\]
For $\lambda\in\sigma(-\Delta_\Omega)$, we set
\[
    E_\lambda(\Omega)
    :=
    \ker(-\Delta_\Omega-\lambda I)
\]
for the corresponding eigenspace. 
The Schr\"odinger propagator is written as
\[
    e^{it\Delta_\Omega}
    =
    e^{-it(-\Delta_\Omega)}.
\]
Unless otherwise stated, $V\subset\Omega$ is measurable with
$|V|>0$. Observation times are denoted by $\tau>0$, and
$C_{V,\tau}$ denotes a corresponding Schr\"odinger observability
constant. Constants denoted by $c_{\Omega,V}$ are allowed to depend
on the fixed domain $\Omega$ and the fixed observation set $V$, but
not on the eigenvalue or on the choice of eigenfunction within an
eigenspace. Finally, $\widetilde u$ denotes a reflected or an extended
eigenfunction.

The remainder of the paper is organized as follows.
Section~\ref{sec:related-work} reviews relevant results on
eigenfunction localization, integrable and rational polygons, and
Schr\"odinger observability.
Section~\ref{sec:quantitative-rectangles} establishes quantitative
full-spectrum non-localization on rectangles, including an explicit
high-frequency threshold and an area-only asymptotic mass floor.
Section~\ref{sec:uniform-L2-nonlocalization} develops the
observability-to-non-localization implication and treats rectangles
and equilateral triangles.
Section~\ref{sec:integrable-polygons} completes the argument for the
four integrable polygonal classes by symmetry reduction and discusses
the associated geometric structure.
Section~\ref{sec:bounded-potentials} extends uniform
non-localization to Dirichlet Schr\"odinger operators with bounded
real-valued potentials on rectangles and isosceles right triangles.
Section~\ref{sec:quasimodes} derives an explicit defect estimate
from observability and applies it to quasimodes and narrow spectral
clusters.
Section~\ref{sec:scope-limitations} discusses the scope and
limitations of the method and possible extensions, and
Section~\ref{sec:conclusion} summarizes the main results.

\section{Related Work}
\label{sec:related-work}

\subsection{Localization of Laplacian Eigenfunctions}
\label{subsec:related-localization}

The localization and spatial structure of Laplacian eigenfunctions
depend strongly on the geometry of the underlying domain; see, for
example, the survey of Grebenkov and Nguyen
\cite{GrebenkovNguyen2013Review}. In their earlier study of
localization in circular, spherical, elliptical, and rectangle-like
domains, Nguyen and Grebenkov \cite{GrebenkovNguyen2013} identified
several mechanisms of high-frequency localization and proved
non-localization results for most rectangle-like domains. They also
formulated the problem of localization in polygonal domains and, more
generally, in piecewise smooth convex domains. 
The present work
addresses the $L^2$ non-localization question for the complete integrable subclass in the trigonometric spectral
classification considered here.

The same non-localization criterion has also been studied for
one-dimensional Schr\"odinger operators. Liard, Lissy, and Privat
\cite{LiardLissyPrivat2018} considered Sturm--Liouville operators with
Dirichlet boundary conditions and investigated uniform lower bounds
for the $L^2$ mass of eigenfunctions on measurable subsets. Their
results include classes of bounded potentials and provide a
one-dimensional counterpart to the non-localization questions
considered here. In contrast, the present work concerns
two-dimensional polygonal domains, where spectral multiplicity and
the geometry of the domain require a different mechanism.

More recently, van den Berg and Bucur
\cite{VanDenBergBucur2025} studied localization of Laplacian
eigenfunctions from a different geometric perspective. Among other
results, they constructed sequences of simply connected planar
polygonal domains for which the corresponding first Dirichlet or
Neumann eigenfunctions exhibit $L^2$ localization. Their setting is
distinct from the fixed-domain high-frequency problem considered
here: localization occurs along a sequence of varying domains, whereas
the present work fixes an integrable polygon and seeks a uniform lower
mass bound throughout its Dirichlet spectrum.

A detailed formulation is as follows. For a
measurable subset $V\subset\Omega$ and $1\leq p<\infty$, one considers
the lower mass ratio
\begin{equation}
    C_p(V;\Omega)
    :=
    \inf_{\lambda\in\sigma(-\Delta_\Omega)}
    \;
    \inf_{\substack{
        u\in E_\lambda(\Omega)\setminus\{0\}
    }}
    \frac{\|u\|_{L^p(V)}}
         {\|u\|_{L^p(\Omega)}}.
    \label{eq:related-Cp}
\end{equation}
Positivity of $C_p(V;\Omega)$ rules out the possibility that
eigenfunctions, including arbitrary choices inside multiple
eigenspaces, asymptotically avoid the set $V$. For domains with simple
or explicitly separated spectra, lower bounds of this type may be
obtained mode by mode. Spectral multiplicity creates an additional
difficulty, however, because estimates for a selected eigenbasis do
not automatically control arbitrary linear combinations within a
degenerate eigenspace.

In this paper, we address this question for $p=2$ on the complete
class of integrable polygons in the trigonometric spectral
classification of McCartin. For every rectangle, isosceles right
triangle, equilateral triangle, and hemi-equilateral triangle, and for
every measurable set $V\subset\Omega$ with $|V|>0$,
\[
    C_2(V;\Omega)>0.
\]
Thus, the localization problem on integrable polygons is treated at
the level of complete eigenspaces rather than through mode-by-mode
estimates.
In this sense, the result gives a complete $L^2$ non-localization
answer for the integrable subclass of the polygonal problem considered
in \cite{GrebenkovNguyen2013}.

The mechanism used here comes from Schr\"odinger observability.
Full-space observability estimates are uniform over all initial data
in $L^2(\Omega)$. When restricted to a Laplacian eigenspace, the
Schr\"odinger evolution acts only by a scalar phase, so the resulting
stationary lower bound is automatically uniform over the entire
eigenspace. This observation is particularly useful in the presence
of spectral multiplicity. The required dynamical estimates are
provided by the observability results of Burq and Zworski in the
rectangular-torus setting and of Alphonse and Lafontaine in the
equilateral-triangle setting, together with the corresponding
reflection and symmetry reductions.

A related stationary perspective has recently been developed by Burq,
Germain, Sorella, and Zhu
\cite{BurqGermainSorellaZhu2026}, studying trace and observability
inequalities for Laplace eigenfunctions on flat tori with respect to
general Borel measures. Their estimates have the form
\[
    \int_{\mathbb{T}^d}|u|^2\,d\mu
    \geq
    c_\mu
    \int_{\mathbb{T}^d}|u|^2\,dx
\]
uniformly over Laplace eigenfunctions. This places stationary
eigenfunction observability in a broader harmonic-analytic framework.
The present work concerns the corresponding non-avoidance problem for
Dirichlet eigenfunctions on integrable polygonal domains and formulates
the conclusion directly through the quantity $C_2(V;\Omega)$.

Quantitative unique continuation and spectral inequalities provide
another closely related line of work. Davey \cite{Davey2020} obtained
quantitative unique continuation estimates for Schr\"odinger operators
with singular lower-order terms. Dicke et al.
\cite{DickeRoseSeelmannTautenhahn2023} established corresponding
estimates for spectral subspaces of Schr\"odinger operators with
singular potentials, while Dicke and Veseli\'c
\cite{DickeVeselic2023} obtained scale-free estimates for gradients of
eigenfunctions of divergence-form operators. More recently, Dicke,
Seelmann, and Veseli\'c \cite{DickeSeelmannVeselic2024} proved spectral
inequalities for Schr\"odinger operators with power-growth confinement
potentials and observation sets of nonuniform density. Such results
typically quantify the dependence of the observation bound on the
spectral scale, whereas the exact-eigenspace estimate considered here
is uniform over the spectrum.

Localization phenomena also occur in singular and thin geometries.
For example, G\'omez, Nazarov, and P\'erez-Mart\'inez
\cite{GomezNazarovPerezMartinez2021} established localization effects
for Dirichlet spectral problems in domains surrounded by thin stiff
and heavy bands, while Cardone, Nazarov, and Taskinen
\cite{CardoneNazarovTaskinen2024} proved localization and exponential
decay of Dirichlet eigenfunctions in a thin-walled beaker. Related
questions concerning concentration and growth of Laplacian
eigenfunctions on compact manifolds have been studied by
Steinerberger \cite{Steinerberger2023}. These mechanisms are
geometrically different from the uniform non-avoidance problem
considered here.

Finally, spectral-subspace estimates provide the natural comparison
for the quasimode results developed later in the paper. The connection
between observability and estimates involving the spectral defect
\[
    \|(A-\lambda)u\|_{L^2}
\]
is classical and is closely related to resolvent and Hautus-type
conditions. In Section~\ref{sec:quasimodes}, we use the time-dependent
observability estimate directly to obtain an explicit defect bound.
Through the spectral theorem, this yields uniform non-localization for
sufficiently accurate quasimodes and for arbitrary vectors in
sufficiently narrow spectral windows. The admissible window width is
determined by the observation time and observability constant, and no
optimality of the resulting threshold is claimed.

\subsection{Rational Polygons and High-Frequency Eigenfunctions}
\label{subsec:rational-polygons}

Rational polygons, whose interior angles are rational multiples of
$\pi$, form a natural class in polygonal spectral theory because their
billiard flows admit finite unfoldings to compact translation
surfaces. A fundamental result of Marklof and Rudnick
\cite{MarklofRudnick2012} shows that, for any rational polygon, a
density-one subsequence of Dirichlet eigenfunctions becomes uniformly
distributed in configuration space.

This result is closely related to, but distinct from, the uniform
non-localization property considered here. Density-one
equidistribution permits exceptional subsequences, whereas the
quantity 
\begin{equation*}
    C_2(V;\Omega)
    =
    \inf_{\lambda\in\sigma(-\Delta_\Omega)}
    \inf_{\substack{u\in E_\lambda(\Omega)\setminus\{0\}}}
    \frac{\|u\|_{L^2(V)}}
         {\|u\|_{L^2(\Omega)}}
    \label{eq:related-uniform}
\end{equation*}
requires a lower bound valid for every eigenspace and every
eigenfunction within it. Thus, the density-one result of
\cite{MarklofRudnick2012} does not by itself imply
$C_2(V;\Omega)>0$.

For the integrable polygons considered below, additional spectral
structure is available. Rectangles admit separated trigonometric
eigenfunctions, while the equilateral triangle has an explicit
trigonometric spectral representation; see, for example,
Pinsky \cite{Pinsky1985}. These structures are closely related to the
reflection and tiling constructions used in the observability
arguments.

\subsection{Schr\"odinger Observability on Flat Tori}
\label{subsec:torus-observability}

Schr\"odinger observability on two-dimensional flat tori has been
developed in several works. Bourgain, Burq, and Zworski
\cite{BourgainBurqZworski2013} established observability results on
rational and irrational two-dimensional tori, including equations
with rough potentials. For the present paper, the key input is the
rough observability result of Burq and Zworski
\cite[Theorem~1.3 and Remark~1]{BurqZworski2019}: on a rectangular
two-dimensional torus, every measurable subset of positive Lebesgue
measure is an observation set for the free Schr\"odinger evolution.

More precisely, for every $T>0$ and every nontrivial
$W\in L^\infty(\mathbb{T}^2)$, their observability estimate controls
the initial $L^2$ norm by
\begin{equation}
    \left\|
        W e^{it\Delta}f
    \right\|_{L^2((0,T)\times\mathbb{T}^2)}.
    \label{eq:related-BZ}
\end{equation}
Taking $W=\mathbf{1}_V$ yields observability from any measurable
positive-measure set $V$. This result is the analytic input used below
after odd reflection of a rectangular Dirichlet eigenfunction to the
associated rectangular flat torus.

The rough-potential theory on two-dimensional tori also provides the
input for the bounded-potential extension considered later in this
paper. Bourgain, Burq, and Zworski
\cite[Theorem~1]{BourgainBurqZworski2013}
proved a stationary estimate for Schr\"odinger operators
\[
    -\Delta+q,
    \qquad
    q\in L^2(\mathbb{T}^2;\mathbb{R}),
\]
on two-dimensional rectangular flat tori, with nonempty open
observation sets. Since the torus has finite measure, their result
applies in particular to real-valued bounded potentials
$q\in L^\infty(\mathbb{T}^2;\mathbb{R})$.
In
Section~\ref{sec:bounded-potentials}, we combine this stationary
estimate with even reflection of the potential and odd reflection of
Dirichlet eigenfunctions to obtain uniform non-localization on
rectangles and isosceles right triangles.

\subsection{Observability on Equilateral Triangles}
\label{subsec:triangle-observability}

Alphonse and Lafontaine
\cite[Theorem~1.2]{AlphonseLafontaine2025} established a corresponding
Schr\"odinger observability result for the equilateral triangle, under
both Dirichlet and Neumann boundary conditions.
Their theorem allows
nonnegative observation weights
\[
    a\in L^2(\mathcal{T})\setminus\{0\},
\]
and therefore applies in particular to characteristic functions of
measurable subsets of positive measure. Their proof exploits the
reflection and tiling structure of the equilateral triangle and a
reduction to a rational twisted-torus setting.

This result provides the second principal observability input used in
the present paper. Once applied to a Laplacian eigenfunction, the
Schr\"odinger evolution changes the eigenfunction only by a unimodular
phase, so the dynamical observability estimate immediately yields a
frequency-uniform lower bound for its spatial $L^2$ mass.

\section{Quantitative Non-Localization on Rectangles}
\label{sec:quantitative-rectangles}

The preceding non-localization question asks only whether the local
$L^2$ mass admits a positive lower bound that is uniform over the
spectrum.  On rectangles, the reflection structure permits a more
quantitative analysis.  In this section we show that, for
finite-perimeter observation sets, sufficiently high-frequency
eigenspaces admit an explicit lower mass bound depending only on the
relative area of the observation set.  The geometry of the
observation set enters only through an explicit frequency threshold.

The argument has three ingredients.  First, the relative area of an
observation set gives a sharp uniform gap for every nontrivial Fourier
coefficient of its normalized indicator measure.  Second, finite
perimeter gives quantitative Fourier decay.  Third, lattice points on
a sufficiently short portion of a high-frequency circle form clusters
of cardinality at most two.  Combining the resulting block
decomposition with a quantitative estimate for the inter-cluster
interaction yields the main high-frequency estimate.

\subsection{Rectangular torus and notation}
\label{subsec:quantitative-rectangle-setting}

Let
\[
    R=(0,a)\times(0,b),
    \qquad a,b>0,
\]
and let
\[
    \mathbb T_{a,b}^2
    :=
    \mathbb R^2/
    \bigl(2a\mathbb Z\times2b\mathbb Z\bigr)
\]
be the associated rectangular flat torus.  Its Euclidean area is
\[
    |\mathbb T_{a,b}^2|=4ab.
\]
The dual lattice is
\begin{equation}
    \Lambda_{a,b}^*
    =
    \frac{\pi}{a}\mathbb Z
    \times
    \frac{\pi}{b}\mathbb Z,
    \label{eq:quant-dual-lattice}
\end{equation}
whose covolume in frequency space is
\begin{equation}
    \kappa_{a,b}
    :=
    \operatorname{covol}(\Lambda_{a,b}^*)
    =
    \frac{\pi^2}{ab}.
    \label{eq:quant-dual-covolume}
\end{equation}
For $r>0$, define the lattice shell
\begin{equation}
    \Lambda_r
    :=
    \left\{
        \xi\in\Lambda_{a,b}^*:
        |\xi|=r
    \right\}.
    \label{eq:quant-lattice-shell}
\end{equation}
Every toral eigenfunction satisfying
\[
    -\Delta U=r^2U
\]
has a finite Fourier representation
\begin{equation}
    U(x)
    =
    \sum_{\xi\in\Lambda_r}
    c_\xi e^{i\xi\cdot x}.
    \label{eq:quant-toral-eigenfunction}
\end{equation}
Orthogonality gives
\begin{equation}
    \|U\|_{L^2(\mathbb T_{a,b}^2)}^2
    =
    4ab
    \sum_{\xi\in\Lambda_r}|c_\xi|^2.
    \label{eq:quant-parseval}
\end{equation}
Let $V\subset\mathbb T_{a,b}^2$ be measurable with
\[
    0<|V|<4ab,
\]
and set
\begin{equation}
    \alpha
    :=
    \frac{|V|}{4ab}
    \in(0,1).
    \label{eq:quant-alpha}
\end{equation}
We associate with $V$ the probability measure
\begin{equation}
    d\mu_V(x)
    :=
    \frac{\mathbf 1_V(x)}{|V|}\,dx.
    \label{eq:quant-muV}
\end{equation}

\subsection{A sharp Fourier gap determined by area}
\label{subsec:sharp-fourier-gap}

We first isolate an extremal fact that depends only on the relative
area $\alpha$.

\begin{lemma}[Sharp Fourier gap for indicator measures]
\label{lem:sharp-fourier-gap}
Let $V\subset\mathbb T_{a,b}^2$ be measurable and let
$\alpha$ be defined by \eqref{eq:quant-alpha}. Then, for every
nonzero dual-lattice frequency $\xi\in\Lambda_{a,b}^*$,
\begin{equation}
    |\widehat{\mu_V}(\xi)|
    \leq
    \frac{\sin(\pi\alpha)}{\pi\alpha}.
    \label{eq:sharp-fourier-gap}
\end{equation}
Consequently,
\begin{equation}
    \eta(\alpha)
    :=
    1-
    \frac{\sin(\pi\alpha)}{\pi\alpha}
    >0.
    \label{eq:eta-alpha}
\end{equation}
The bound \eqref{eq:sharp-fourier-gap} is sharp for each fixed
nontrivial character.
\end{lemma}

\begin{proof}
Let
\[
    dm(x)
    :=
    \frac{dx}{4ab}
\]
denote normalized Haar measure on $\mathbb T_{a,b}^2$. Then
\[
    d\mu_V
    =
    \alpha^{-1}\mathbf1_V\,dm.
\]
Fix $\xi\neq0$ and set
$\chi_\xi(x):=e^{i\xi\cdot x}$. 
The nontrivial character $\chi_\xi$ is a continuous surjective
homomorphism from $\mathbb T_{a,b}^2$ onto $\mathbb S^1$.
Accordingly, the pushforward of normalized Haar measure under
$\chi_\xi$ is normalized Haar measure on $\mathbb S^1$.

Choose $\theta\in\mathbb R$ such that
$e^{i\theta}
    \int_V e^{-i\xi\cdot x}\,dm(x)$
is nonnegative and real. Then
\begin{align}
    \left|
        \int_Ve^{-i\xi\cdot x}\,dm(x)
    \right|
    &=
    \int_V
    \cos(\xi\cdot x-\theta)\,dm(x).
    \label{eq:fourier-gap-real-part}
\end{align}
Since $\xi\cdot x-\theta$ modulo $2\pi$ is uniformly distributed
under Haar measure, the bathtub principle implies that among all
measurable subsets of measure $\alpha$, the right-hand side is
maximized by selecting the $\alpha$-fraction of phases on which
$\cos t$ is largest. Hence
\begin{align}
    \left|
        \int_Ve^{-i\xi\cdot x}\,dm(x)
    \right|
    &\leq
    \frac{1}{2\pi}
    \int_{-\pi\alpha}^{\pi\alpha}
    \cos t\,dt
    \notag\\
    &=
    \frac{\sin(\pi\alpha)}{\pi}.
    \label{eq:fourier-gap-extremal}
\end{align}
Dividing by $\alpha$ yields
\eqref{eq:sharp-fourier-gap}. Finally,
\[
    \sin(\pi\alpha)<\pi\alpha,
    \qquad 0<\alpha<1,
\]
so $\eta(\alpha)>0$.
\end{proof}

\begin{remark}
\label{rem:fourier-gap-small-alpha}
As $\alpha\downarrow0$,
\begin{equation}
    \eta(\alpha)
    =
    \frac{\pi^2}{6}\alpha^2
    +O(\alpha^4).
    \label{eq:eta-small-alpha}
\end{equation}
Thus the area gap is quadratic in the small-area regime.
\end{remark}

\subsection{Fourier decay from finite perimeter}
\label{subsec:perimeter-fourier-decay}

We next quantify the decay of the Fourier coefficients when $V$ has
finite perimeter on the torus.  We write
\[
    P_{\mathbb T}(V)
    :=
    |D\mathbf1_V|(\mathbb T_{a,b}^2).
\]

\begin{lemma}[Finite-perimeter Fourier decay]
\label{lem:finite-perimeter-fourier}
Assume that $V\subset\mathbb T_{a,b}^2$ has finite perimeter.
Then, for every
$\xi\in\Lambda_{a,b}^*\setminus\{0\}$,
\begin{equation}
    \left|
        \int_Ve^{-i\xi\cdot x}\,dx
    \right|
    \leq
    \frac{P_{\mathbb T}(V)}{|\xi|}.
    \label{eq:perimeter-fourier-indicator}
\end{equation}
Equivalently,
\begin{equation}
    |\widehat{\mu_V}(\xi)|
    \leq
    \frac{P_{\mathbb T}(V)}
         {|V|\,|\xi|}.
    \label{eq:perimeter-fourier-mu}
\end{equation}
\end{lemma}

\begin{proof}
Since $V$ has finite perimeter,
$\mathbf1_V\in BV(\mathbb T_{a,b}^2)$ and
\[
    |D\mathbf1_V|(\mathbb T_{a,b}^2)
    =
    P_{\mathbb T}(V).
\]
Fix $\xi\neq0$ and write
\[
    e=\frac{\xi}{|\xi|}.
\]
For
\[
    \varphi(x)=e^{-i\xi\cdot x},
\]
one has
\[
    \partial_e\varphi
    =
    -i|\xi|\varphi.
\]
The definition of the distributional derivative therefore gives
\[
    -i|\xi|
    \int_{\mathbb T_{a,b}^2}
    \mathbf1_V(x)e^{-i\xi\cdot x}\,dx
    =
    -
    \int_{\mathbb T_{a,b}^2}
    e^{-i\xi\cdot x}\,
    e\cdot D\mathbf1_V(x).
\]
Taking absolute values yields
\[
    |\xi|
    \left|
        \int_Ve^{-i\xi\cdot x}\,dx
    \right|
    \leq
    |D\mathbf1_V|(\mathbb T_{a,b}^2),
\]
which proves \eqref{eq:perimeter-fourier-indicator}.
Division by $|V|$ gives
\eqref{eq:perimeter-fourier-mu}.
\end{proof}

\subsection{Explicit two-point clustering on lattice circles}
\label{subsec:two-point-clustering}

We now exploit the geometry of the dual lattice
$\Lambda_{a,b}^*$.

\begin{lemma}[Explicit two-point lattice clustering]
\label{lem:explicit-two-point-clustering}
Let $\Lambda\subset\mathbb R^2$ be a rank-two lattice of covolume
$\kappa$, and let
\[
    \Lambda_r
    =
    \Lambda\cap\{|\xi|=r\}.
\]
Define
\begin{equation}
    D_r
    :=
    \left(\frac{\kappa r}{2}\right)^{1/3}.
    \label{eq:cluster-distance}
\end{equation}
Construct a graph on $\Lambda_r$ by joining two distinct points when
their Euclidean distance is strictly smaller than $D_r$.

Then every connected component of this graph contains at most two
points. Distinct connected components have mutual Euclidean distance
at least $D_r$.
\end{lemma}

\begin{proof}
The separation of distinct components follows directly from the
definition.

Suppose that a connected component contains at least three points.
Then there exist three distinct points
$\xi_1,\xi_2,\xi_3$ in the component such that
\[
    |\xi_1-\xi_2|<D_r,
    \qquad
    |\xi_2-\xi_3|<D_r.
\]
Consequently,
\[
    |\xi_1-\xi_3|<2D_r.
\]
The three points lie on the same Euclidean circle and are therefore
noncollinear. Since they belong to $\Lambda$, the Euclidean area
$\mathcal A$ of the lattice triangle they determine satisfies
\begin{equation}
    \mathcal A
    \geq
    \frac{\kappa}{2}.
    \label{eq:lattice-triangle-lower-area}
\end{equation}
On the other hand, the circumradius of the triangle equals $r$.
Writing its three side lengths as $a_1,a_2,a_3$, the circumradius
formula gives
\[
    \mathcal A
    =
    \frac{a_1a_2a_3}{4r}.
\]
Hence
\[
    \mathcal A
    <
    \frac{D_r\cdot D_r\cdot2D_r}{4r}
    =
    \frac{D_r^3}{2r}
    =
    \frac{\kappa}{4},
\]
contradicting \eqref{eq:lattice-triangle-lower-area}.
Thus every component has cardinality at most two.
\end{proof}

For the rectangular dual lattice
\eqref{eq:quant-dual-lattice}, this gives
\begin{equation}
    D_r
    =
    \left(
        \frac{\pi^2r}{2ab}
    \right)^{1/3}.
    \label{eq:rectangle-cluster-distance}
\end{equation}

\subsection{Inter-cluster interactions}
\label{subsec:inter-cluster-interactions}

The following elementary packing estimate will be used to sum the
Fourier interactions between distinct clusters.

\begin{lemma}[Reciprocal-distance packing on a circle]
\label{lem:reciprocal-circle-packing}
Let $Y$ be a $D$-separated subset of the Euclidean circle $S_r$,
and let $\xi\in S_r$ satisfy
\[
    \operatorname{dist}(\xi,Y)\geq D.
\]
Assume $0<D\leq r$. Then
\begin{equation}
    \sum_{\eta\in Y}
    \frac{1}{|\xi-\eta|}
    \leq
    \frac{\pi}{D}
    \left[
        1+
        \log\left(
            1+\frac{\pi r}{D}
        \right)
    \right].
    \label{eq:reciprocal-packing}
\end{equation}
\end{lemma}

\begin{proof}
For $\eta\in S_r$, let $s(\xi,\eta)\in[0,\pi r]$ denote the shorter
arc distance between $\xi$ and $\eta$. Then
\[
    |\xi-\eta|
    =
    2r
    \sin\left(
        \frac{s(\xi,\eta)}{2r}
    \right).
\]
Since
\[
    \sin t\geq\frac{2t}{\pi},
    \qquad 0\leq t\leq\frac{\pi}{2},
\]
we obtain
\begin{equation}
    |\xi-\eta|
    \geq
    \frac{2}{\pi}s(\xi,\eta).
    \label{eq:chord-arc-lower}
\end{equation}

Order the points of $Y$ on each of the two semicircles emanating from
$\xi$ by increasing arc distance. Since Euclidean separation by $D$
implies arc separation by at least $D$, the $j$th point on either
side has arc distance at least $jD$. Hence
\[
    \frac{1}{|\xi-\eta_j|}
    \leq
    \frac{\pi}{2jD}.
\]
Summing over the two semicircles and using
\[
    H_N\leq1+\log N
\]
gives \eqref{eq:reciprocal-packing}.
\end{proof}

\subsection{Quantitative high-frequency non-localization}
\label{subsec:quantitative-high-frequency}

We now combine the preceding ingredients.
For the shell $\Lambda_r$, let
\[
    \Lambda_r
    =
    \bigsqcup_{\beta\in\mathcal B_r}
    \Omega_\beta
\]
denote the connected components determined by
\cref{lem:explicit-two-point-clustering}. Thus,
\[
    |\Omega_\beta|\leq2
\]
and distinct clusters are separated by at least $D_r$.
We define
\begin{equation}
    \varepsilon_V(r)
    :=
    \frac{2\pi P_{\mathbb T}(V)}
         {|V|D_r}
    \left[
        1+
        \log\left(
            1+\frac{\pi r}{D_r}
        \right)
    \right].
    \label{eq:epsilonVr}
\end{equation}

\begin{theorem}[Quantitative high-frequency non-localization]
\label{thm:quantitative-high-frequency}
Let $V\subset\mathbb T_{a,b}^2$ be measurable with finite perimeter
and
\[
    0<|V|<4ab.
\]
Let
\[
    \alpha=\frac{|V|}{4ab}.
\]
Assume that
\begin{equation}
    r^2\geq\frac{\kappa_{a,b}}{2}.
    \label{eq:packing-frequency-condition}
\end{equation}
Then every toral eigenfunction
\[
    -\Delta U=r^2U
\]
satisfies
\begin{equation}
    \frac{\|U\|_{L^2(V)}^2}
         {\|U\|_{L^2(\mathbb T_{a,b}^2)}^2}
    \geq
    \alpha
    \left[
        \eta(\alpha)-\varepsilon_V(r)
    \right],
    \label{eq:quantitative-high-frequency-general}
\end{equation}
whenever the quantity in brackets is positive.

In particular, if
\begin{equation}
    \varepsilon_V(r)
    \leq
    \frac{\eta(\alpha)}{2},
    \label{eq:high-frequency-condition}
\end{equation}
then
\begin{equation}
{
    \frac{\|U\|_{L^2(V)}}
         {\|U\|_{L^2(\mathbb T_{a,b}^2)}}
    \geq
    \left[
        \frac{\alpha}{2}
        \left(
            1-
            \frac{\sin(\pi\alpha)}{\pi\alpha}
        \right)
    \right]^{1/2}.
    }
    \label{eq:quantitative-high-frequency-final}
\end{equation}
The lower bound in
\eqref{eq:quantitative-high-frequency-final} is independent of the
eigenvalue, its multiplicity, and the geometry of $V$ beyond its area
fraction.
\end{theorem}

\begin{proof}
Write
\[
    U(x)
    =
    \sum_{\xi\in\Lambda_r}
    c_\xi e^{i\xi\cdot x}.
\]
With respect to the probability measure $\mu_V$,
\begin{equation}
    \int_{\mathbb T_{a,b}^2}
    |U|^2\,d\mu_V
    =
    \sum_{\xi,\eta\in\Lambda_r}
    c_\xi\overline{c_\eta}
    \widehat{\mu_V}(\eta-\xi).
    \label{eq:observation-Gram}
\end{equation}
We split the corresponding Gram matrix as
\[
    G_V
    =
    G_{\mathrm{block}}+B_r,
\]
where $G_{\mathrm{block}}$ retains only interactions between
frequencies belonging to the same cluster.
If
\[
    \Omega_\beta=\{\xi,\eta\},
\]
its block is
\[
    \begin{pmatrix}
        1 & \widehat{\mu_V}(\eta-\xi)\\
        \overline{\widehat{\mu_V}(\eta-\xi)} & 1
    \end{pmatrix}.
\]
The smallest eigenvalue of this matrix equals
\[
    1-
    |\widehat{\mu_V}(\eta-\xi)|.
\]
By Theorem \ref{lem:sharp-fourier-gap},
\[
    G_{\mathrm{block}}
    \geq
    \eta(\alpha)I.
    \label{eq:block-lower}
\]
It remains to estimate $B_r$. By Theorem
\ref{lem:finite-perimeter-fourier},
\[
    |(B_r)_{\xi\eta}|
    \leq
    \frac{P_{\mathbb T}(V)}
         {|V|\,|\xi-\eta|}
\]
whenever $\xi$ and $\eta$ belong to distinct clusters.

Since each cluster contains at most two points, the set of
frequencies outside the cluster containing a fixed $\xi$ can be
partitioned into two $D_r$-separated sets. Applying
\cref{lem:reciprocal-circle-packing} to each set yields
\[
    \sup_{\xi\in\Lambda_r}
    \sum_{\eta\in\Lambda_r}
    |(B_r)_{\xi\eta}|
    \leq
    \varepsilon_V(r).
\]
Since $B_r$ is Hermitian, Schur's test gives
\begin{equation}
    \|B_r\|_{\ell^2\to\ell^2}
    \leq
    \varepsilon_V(r).
    \label{eq:Schur-Br}
\end{equation}
Combining
\eqref{eq:block-lower} and \eqref{eq:Schur-Br},
\[
    G_V
    \geq
    \bigl(
        \eta(\alpha)-\varepsilon_V(r)
    \bigr)I.
\]
Hence
\[
    \int|U|^2\,d\mu_V
    \geq
    \bigl(
        \eta(\alpha)-\varepsilon_V(r)
    \bigr)
    \sum_{\xi\in\Lambda_r}|c_\xi|^2.
\]
Using
\[
    \int|U|^2\,d\mu_V
    =
    \frac{1}{|V|}
    \|U\|_{L^2(V)}^2
\]
and \eqref{eq:quant-parseval}, we obtain
\[
    \frac{\|U\|_{L^2(V)}^2}
         {\|U\|_{L^2(\mathbb T_{a,b}^2)}^2}
    \geq
    \frac{|V|}{4ab}
    \bigl(
        \eta(\alpha)-\varepsilon_V(r)
    \bigr),
\]
which is \eqref{eq:quantitative-high-frequency-general}. The final
claim follows from
\eqref{eq:high-frequency-condition}.
\end{proof}

\subsection{Explicit high-frequency threshold and asymptotic universality}
\label{subsec:explicit-high-frequency-threshold}

The estimate in
\cref{thm:quantitative-high-frequency} becomes effective once
\[
    \varepsilon_V(r)
    \leq
    \frac{\eta(\alpha)}{2}.
\]
We now replace this implicit condition by a closed-form sufficient
frequency threshold.

Recall that
\[
    \kappa_{a,b}
    =
    \frac{\pi^2}{ab}.
\]
Introduce the dimensionless frequency
\begin{equation}
    \rho
    :=
    \frac{r}{\sqrt{\kappa_{a,b}}},
    \label{eq:dimensionless-frequency}
\end{equation}
and the dimensionless perimeter parameter
\begin{equation}
    \chi_V
    :=
    \frac{P_{\mathbb T}(V)}
         {|V|\sqrt{\kappa_{a,b}}}.
    \label{eq:dimensionless-perimeter}
\end{equation}
Since
\[
    D_r
    =
    \left(
        \frac{\kappa_{a,b}r}{2}
    \right)^{1/3},
\]
we may write
\begin{equation}
    D_r
    =
    \sqrt{\kappa_{a,b}}
    \left(
        \frac{\rho}{2}
    \right)^{1/3}.
    \label{eq:Dr-dimensionless}
\end{equation}
Consequently, the error term
\eqref{eq:epsilonVr} becomes
\begin{equation}
    \varepsilon_V(r)
    =
    2\pi 2^{1/3}
    \chi_V
    \rho^{-1/3}
    \left[
        1+
        \log\left(
            1+
            \pi 2^{1/3}\rho^{2/3}
        \right)
    \right].
    \label{eq:epsilon-dimensionless}
\end{equation}

We use the elementary inequality
\begin{equation}
    \log(1+x)
    \leq
    4x^{1/4},
    \qquad x\geq0.
    \label{eq:log-elementary-bound}
\end{equation}
If $\rho\geq1$, then
\[
    1\leq\rho^{1/6},
\]
and hence
\begin{align}
    1+
    \log\left(
        1+\pi2^{1/3}\rho^{2/3}
    \right)
    &\leq
    1+
    4\left(
        \pi2^{1/3}
    \right)^{1/4}
    \rho^{1/6}
    \notag\\
    &\leq
    \left[
        1+
        4\left(
            \pi2^{1/3}
        \right)^{1/4}
    \right]
    \rho^{1/6}.
    \label{eq:log-frequency-bound}
\end{align}
Define
\begin{equation}
    C_0
    :=
    2\pi2^{1/3}
    \left[
        1+
        4\left(
            \pi2^{1/3}
        \right)^{1/4}
    \right].
    \label{eq:C0}
\end{equation}
Then
\begin{equation}
{
    \varepsilon_V(r)
    \leq
    C_0\chi_V\rho^{-1/6},
    \qquad
    \rho\geq1.
    }
    \label{eq:epsilon-explicit-upper}
\end{equation}

This gives an explicit sufficient threshold.

\begin{proposition}[Explicit toral frequency threshold]
\label{prop:explicit-toral-threshold}
Let $V\subset\mathbb T_{a,b}^2$ satisfy the assumptions of
\cref{thm:quantitative-high-frequency}. Define
\begin{equation}
    \rho_*(V)
    :=
    \max
    \left\{
        1,\,
        \left(
            \frac{
                2C_0\chi_V
            }{
                \eta(\alpha)
            }
        \right)^6
    \right\}.
    \label{eq:rho-star-torus}
\end{equation}
Then every $r$ satisfying
\begin{equation}
    r
    \geq
    r_*(V)
    :=
    \sqrt{\kappa_{a,b}}\,
    \rho_*(V)
    \label{eq:r-star-torus}
\end{equation}
obeys
\begin{equation}
    \varepsilon_V(r)
    \leq
    \frac{\eta(\alpha)}{2}.
    \label{eq:epsilon-half-eta-explicit}
\end{equation}
Consequently, every toral eigenfunction
\[
    -\Delta U=r^2U,
    \qquad
    r\geq r_*(V),
\]
satisfies
\begin{equation}
    \frac{\|U\|_{L^2(V)}}
         {\|U\|_{L^2(\mathbb T_{a,b}^2)}}
    \geq
    F_{\mathrm{HF}}(\alpha),
    \label{eq:explicit-torus-HF}
\end{equation}
where
\begin{equation}
    F_{\mathrm{HF}}(\alpha)
    :=
    \left[
        \frac{\alpha}{2}
        \left(
            1-
            \frac{\sin(\pi\alpha)}
                 {\pi\alpha}
        \right)
    \right]^{1/2}.
    \label{eq:FHF}
\end{equation}
\end{proposition}

\begin{proof}
If $r\geq r_*(V)$, then
\[
    \rho
    =
    \frac{r}{\sqrt{\kappa_{a,b}}}
    \geq
    \rho_*(V)
    \geq1.
\]
Thus \eqref{eq:epsilon-explicit-upper} applies. Moreover,
\[
    \rho^{-1/6}
    \leq
    \frac{\eta(\alpha)}
         {2C_0\chi_V},
\]
and therefore
\[
    \varepsilon_V(r)
    \leq
    C_0\chi_V\rho^{-1/6}
    \leq
    \frac{\eta(\alpha)}{2}.
\]
The conclusion now follows from
\cref{thm:quantitative-high-frequency}.
\end{proof}

\begin{remark}[Implicit versus closed-form thresholds]
\label{rem:implicit-explicit-threshold}
The sharper implicit threshold associated with the present estimate may be defined
implicitly as the smallest frequency above which
\[
    \varepsilon_V(r)
    \leq
    \frac{\eta(\alpha)}{2}.
\]
Proposition~\ref{prop:explicit-toral-threshold} replaces this
transcendental condition by a closed-form sufficient threshold.  The
latter is convenient for quantitative statements but is not expected
to be optimal.
\end{remark}

For small observation sets,
\begin{equation}
    F_{\mathrm{HF}}(\alpha)
    =
    \frac{\pi}{\sqrt{12}}\,
    \alpha^{3/2}
    +
    O(\alpha^{7/2}),
    \qquad
    \alpha\downarrow0.
    \label{eq:FHF-small-alpha}
\end{equation}

\subsection{Transfer to Dirichlet rectangles}
\label{subsec:quantitative-transfer-rectangle}

Let $V\subset R$, and let
$V^\sharp\subset\mathbb T_{a,b}^2$ denote its fourfold reflected
extension. Then
\begin{equation}
    \frac{|V^\sharp|}
         {|\mathbb T_{a,b}^2|}
    =
    \frac{|V|}{|R|}
    =:\alpha.
    \label{eq:reflected-area-fraction}
\end{equation}
If $u$ is a Dirichlet eigenfunction on $R$ and $\widetilde u$ denotes
its odd periodic extension to $\mathbb T_{a,b}^2$, then
\begin{equation}
    \frac{
        \|\widetilde u\|_{L^2(V^\sharp)}^2
    }{
        \|\widetilde u\|_{L^2(\mathbb T_{a,b}^2)}^2
    }
    =
    \frac{
        \|u\|_{L^2(V)}^2
    }{
        \|u\|_{L^2(R)}^2
    }.
    \label{eq:reflected-ratio-identity}
\end{equation}
We define the reflected perimeter by
\begin{equation}
    P_{\mathrm{ref}}(V)
    :=
    P_{\mathbb T}(V^\sharp).
    \label{eq:reflected-perimeter}
\end{equation}
This formulation avoids imposing additional trace assumptions on
$V$ along $\partial R$.
For the reflected set, introduce
\begin{equation}
    \chi_{\mathrm{ref}}(V)
    :=
    \frac{
        P_{\mathrm{ref}}(V)
    }{
        |V^\sharp|
        \sqrt{\kappa_{a,b}}
    }.
    \label{eq:chi-ref}
\end{equation}
Since
\[
    |V^\sharp|=4|V|,
    \qquad
    \sqrt{\kappa_{a,b}}
    =
    \frac{\pi}{\sqrt{ab}},
\]
we have the explicit identity
\begin{equation}
    \chi_{\mathrm{ref}}(V)
    =
    \frac{
        P_{\mathrm{ref}}(V)\sqrt{ab}
    }{
        4\pi|V|
    }.
    \label{eq:chi-ref-explicit}
\end{equation}
We can now state the quantitative rectangle theorem with a completely
explicit sufficient spectral threshold.

\begin{theorem}[Explicit high-frequency non-localization on rectangles]
\label{thm:explicit-high-frequency-rectangle}
Let
\[
    R=(0,a)\times(0,b),
\]
and let $V\subset R$ be measurable such that its reflected extension
$V^\sharp$ has finite perimeter and $0<|V|<|R|$. Set
\begin{equation}
    \alpha
    :=
    \frac{|V|}{|R|}
    =
    \frac{|V|}{ab},
    \qquad
    \eta(\alpha)
    :=
    1-
    \frac{\sin(\pi\alpha)}
         {\pi\alpha}.
    \label{eq:rectangle-alpha-eta}
\end{equation}
Define
\begin{equation}
    \rho_*(R,V)
    :=
    \max
    \left\{
        1,\,
        \left(
            \frac{
                2C_0
                \chi_{\mathrm{ref}}(V)
            }{
                \eta(\alpha)
            }
        \right)^6
    \right\}
    \label{eq:rho-star-rectangle}
\end{equation}
and
\begin{equation}
    \lambda_*(R,V)
    :=
    \frac{\pi^2}{ab}
    \rho_*(R,V)^2.
    \label{eq:lambda-star-rectangle}
\end{equation}
Equivalently,
\begin{equation}
{
    \lambda_*(R,V)
    =
    \frac{\pi^2}{ab}
    \max
    \left\{
        1,\,
        \left[
            \frac{
                C_0
                P_{\mathrm{ref}}(V)
                \sqrt{ab}
            }{
                2\pi
                |V|
                \eta(\alpha)
            }
        \right]^{12}
    \right\}.
    }
    \label{eq:lambda-star-explicit}
\end{equation}

Then every Dirichlet Laplacian eigenfunction
\begin{equation}
    -\Delta_Ru=\lambda u,
    \qquad
    \lambda\geq\lambda_*(R,V),
    \label{eq:rectangle-eigenvalue-threshold}
\end{equation}
satisfies
\begin{equation}
{
    \frac{\|u\|_{L^2(V)}}
         {\|u\|_{L^2(R)}}
    \geq
    \left[
        \frac{\alpha}{2}
        \left(
            1-
            \frac{\sin(\pi\alpha)}
                 {\pi\alpha}
        \right)
    \right]^{1/2}.
    }
    \label{eq:rectangle-area-only-final}
\end{equation}
The estimate is uniform in the eigenvalue, its multiplicity, and the
choice of eigenfunction within the corresponding eigenspace.
\end{theorem}

\begin{proof}
Let $\widetilde u$ be the odd periodic extension of $u$ to
$\mathbb T_{a,b}^2$, and write
\[
    r=\sqrt{\lambda},
    \qquad
    \rho=\frac{r}{\sqrt{\kappa_{a,b}}}.
\]
The condition
\[
    \lambda\geq\lambda_*(R,V)
\]
implies
\[
    \rho\geq\rho_*(R,V).
\]
Applying
\cref{prop:explicit-toral-threshold} to $V^\sharp$ gives
\[
    \frac{
        \|\widetilde u\|_{L^2(V^\sharp)}
    }{
        \|\widetilde u\|_{L^2(\mathbb T_{a,b}^2)}
    }
    \geq
    F_{\mathrm{HF}}(\alpha).
\]
The reflection identity
\eqref{eq:reflected-ratio-identity} then yields
\eqref{eq:rectangle-area-only-final}.
\end{proof}

\begin{remark}[Scaling consistency]
\label{rem:threshold-scaling}
The threshold has the natural Laplacian scaling.  If both the
rectangle and observation set are dilated by a factor $s>0$, then
$\alpha$ and $\chi_{\mathrm{ref}}$ remain unchanged, while
\[
    \frac{\pi^2}{ab}
    \longmapsto
    s^{-2}\frac{\pi^2}{ab}.
\]
Consequently,
\begin{equation}
    \lambda_*(sR,sV)
    =
    s^{-2}\lambda_*(R,V).
    \label{eq:threshold-scaling}
\end{equation}
\end{remark}

\begin{remark}[Nonoptimality of the closed-form threshold]
\label{rem:threshold-nonoptimal}
The threshold
\eqref{eq:lambda-star-explicit} is a sufficient explicit threshold
rather than an optimized one.  In particular, the power $12$ arises
from replacing the logarithmic factor in
\eqref{eq:epsilon-dimensionless} by the elementary power bound
\eqref{eq:log-elementary-bound}.  A substantially smaller threshold
is obtained by solving directly the implicit condition
\begin{equation}
    \varepsilon_{V^\sharp}(\sqrt{\lambda})
    \leq
    \frac{\eta(\alpha)}{2}.
    \label{eq:implicit-sharp-threshold}
\end{equation}
No optimality of the spectral threshold is claimed here.
\end{remark}

\subsection{Low-Frequency Quantitative Non-Localization}
\label{subsec:low-frequency-quantitative}

It remains to control the finite spectral regime below the explicit
threshold $\lambda_*(R,V)$. Unlike the high-frequency argument, which
exploits the geometry of lattice points on a frequency circle, the
low-frequency regime can be treated by a finite-bandwidth uncertainty
principle.

We use the following $L^2$ form of the Tur\'an--Nazarov inequality.
Nazarov's theorem~\cite{Nazarov2000Turan} implies, in the form recorded
by Tao~\cite[Lemma~13]{Tao2021ExactControl}, that there exists a
universal constant $C_{\mathrm{NT}}\geq1$ such that every
trigonometric polynomial
\[
    p(t)
    =
    \sum_{k=0}^{M-1} c_k e^{2\pi i k t}
\]
satisfies
\begin{equation}
    \|p\|_{L^2(\mathbb T)}
    \leq
    \left(
        \frac{C_{\mathrm{NT}}}{|E|}
    \right)^{M-\frac12}
    \|p\|_{L^2(E)}
    \label{eq:NT-one-dimensional}
\end{equation}
for every measurable $E\subset\mathbb T$ with $|E|>0$.

Applying this estimate successively along
the two coordinate directions gives the following finite-bandwidth
observation inequality.

\begin{lemma}[Finite-bandwidth observation on the two-torus]
\label{lem:finite-bandwidth-observation}
Let
\[
    \mathbb T^2
    =
    \mathbb R^2/\mathbb Z^2
\]
be equipped with normalized Haar measure, and let
$W\subset\mathbb T^2$ be measurable with
\[
    |W|=\alpha\in(0,1).
\]
Suppose that
\begin{equation}
    p(x,y)
    =
    \sum_{\substack{|m|\leq N_x\\|n|\leq N_y}}
    c_{m,n}e^{2\pi i(mx+ny)}.
    \label{eq:finite-bandwidth-polynomial}
\end{equation}
Then
\begin{equation}
{
    \|p\|_{L^2(W)}
    \geq
    \left(
        \frac{\alpha}{2C_{\mathrm{NT}}}
    \right)^{
        2N_x+2N_y+1
    }
    \|p\|_{L^2(\mathbb T^2)}.
    }
    \label{eq:finite-bandwidth-observation}
\end{equation}
\end{lemma}

\begin{proof}
For $y\in\mathbb T$, define the horizontal section
\[
    W_y
    :=
    \{x\in\mathbb T:(x,y)\in W\},
\]
and set
\[
    f(y):=|W_y|.
\]
By Fubini's theorem,
\[
    \int_{\mathbb T}f(y)\,dy=\alpha.
\]
Define
\begin{equation}
    Y
    :=
    \left\{
        y\in\mathbb T:
        |W_y|\geq\frac{\alpha}{2}
    \right\}.
    \label{eq:good-horizontal-slices}
\end{equation}
Since $0\leq f\leq1$,
\[
    \alpha
    =
    \int_Yf(y)\,dy
    +
    \int_{\mathbb T\setminus Y}f(y)\,dy
    \leq
    |Y|
    +
    \frac{\alpha}{2}(1-|Y|).
\]
Therefore
\begin{equation}
    |Y|
    \geq
    \frac{\alpha}{2-\alpha}
    \geq
    \frac{\alpha}{2}.
    \label{eq:Y-lower-measure}
\end{equation}
Fix $y\in Y$. As a function of $x$, the polynomial
$p(\cdot,y)$ contains at most
\[
    M_x:=2N_x+1
\]
distinct characters. 
Multiplication by the unimodular factor
$e^{2\pi iN_xx}$ shifts these frequencies into
$\{0,\ldots,2N_x\}$ without changing either of the relevant
$L^2$ norms. Thus \eqref{eq:NT-one-dimensional} applies with
\[
    M_x=2N_x+1.
\]
Applying
\eqref{eq:NT-one-dimensional} to the measurable set $W_y$ gives
\begin{align}
    \|p(\cdot,y)\|_{L^2(\mathbb T)}
    &\leq
    \left(
        \frac{C_{\mathrm{NT}}}{|W_y|}
    \right)^{M_x-\frac12}
    \|p(\cdot,y)\|_{L^2(W_y)}
    \notag\\
    &\leq
    \left(
        \frac{2C_{\mathrm{NT}}}{\alpha}
    \right)^{M_x-\frac12}
    \|p(\cdot,y)\|_{L^2(W_y)}.
    \label{eq:first-fiber-NT}
\end{align}
Squaring and integrating over $y\in Y$ yields
\begin{equation}
    \|p\|_{L^2(\mathbb T\times Y)}
    \leq
    \left(
        \frac{2C_{\mathrm{NT}}}{\alpha}
    \right)^{M_x-\frac12}
    \|p\|_{L^2(W)}.
    \label{eq:first-fiber-integrated}
\end{equation}
Next, fix $x\in\mathbb T$. As a function of $y$, the polynomial
$p(x,\cdot)$ contains at most
\[
    M_y:=2N_y+1
\]
distinct characters. Applying
\eqref{eq:NT-one-dimensional} to $Y$, together with
\eqref{eq:Y-lower-measure}, gives
\[
    \|p(x,\cdot)\|_{L^2(\mathbb T)}
    \leq
    \left(
        \frac{2C_{\mathrm{NT}}}{\alpha}
    \right)^{M_y-\frac12}
    \|p(x,\cdot)\|_{L^2(Y)}.
\]
Integrating in $x$, we obtain
\begin{equation}
    \|p\|_{L^2(\mathbb T^2)}
    \leq
    \left(
        \frac{2C_{\mathrm{NT}}}{\alpha}
    \right)^{M_y-\frac12}
    \|p\|_{L^2(\mathbb T\times Y)}.
    \label{eq:second-fiber-integrated}
\end{equation}
Combining
\eqref{eq:first-fiber-integrated} and
\eqref{eq:second-fiber-integrated} yields
\[
    \|p\|_{L^2(\mathbb T^2)}
    \leq
    \left(
        \frac{2C_{\mathrm{NT}}}{\alpha}
    \right)^{
        M_x+M_y-1
    }
    \|p\|_{L^2(W)}.
\]
Since
\[
    M_x+M_y-1
    =
    2N_x+2N_y+1,
\]
the result follows.
\end{proof}

The preceding lemma immediately yields a quantitative estimate for
all rectangular-torus eigenfunctions below a prescribed spectral
threshold.

\begin{proposition}[Explicit low-frequency bound on the rectangular torus]
\label{prop:explicit-low-frequency-torus}
Let $W\subset\mathbb T_{a,b}^2$ be measurable with relative area
\[
    \alpha
    =
    \frac{|W|}{4ab},
\]
and let $\Lambda_*>0$. Define
\begin{equation}
    N_x(\Lambda_*)
    :=
    \left\lfloor
        \frac{a\sqrt{\Lambda_*}}{\pi}
    \right\rfloor,
    \qquad
    N_y(\Lambda_*)
    :=
    \left\lfloor
        \frac{b\sqrt{\Lambda_*}}{\pi}
    \right\rfloor.
    \label{eq:low-frequency-Nxy}
\end{equation}
Then every toral eigenfunction
\[
    -\Delta U=\lambda U,
    \qquad
    0\leq\lambda<\Lambda_*,
\]
satisfies
\begin{equation}
{
    \frac{\|U\|_{L^2(W)}}
         {\|U\|_{L^2(\mathbb T_{a,b}^2)}}
    \geq
    F_{\mathrm{LF}}(\alpha,\Lambda_*;a,b),
    }
    \label{eq:low-frequency-torus-final}
\end{equation}
where
\begin{equation}
    F_{\mathrm{LF}}(\alpha,\Lambda_*;a,b)
    :=
    \left(
        \frac{\alpha}{2C_{\mathrm{NT}}}
    \right)^{
        2N_x(\Lambda_*)
        +
        2N_y(\Lambda_*)
        +1
    }.
    \label{eq:FLF-general}
\end{equation}
\end{proposition}

\begin{proof}

Define
\[
    P(X,Y)
    :=
    U(2aX,2bY),
    \qquad
    (X,Y)\in\mathbb T^2,
\]
and let
\[
    W'
    :=
    \left\{
        (X,Y)\in\mathbb T^2:
        (2aX,2bY)\in W
    \right\}.
\]
Then
\[
    |W'|
    =
    \frac{|W|}{4ab}
    =
    \alpha.
\]
Moreover, by the change of variables
\[
    x=2aX,
    \qquad
    y=2bY,
\]
we have
\begin{equation}
    \frac{\|P\|_{L^2(W')}}
         {\|P\|_{L^2(\mathbb T^2)}}
    =
    \frac{\|U\|_{L^2(W)}}
         {\|U\|_{L^2(\mathbb T_{a,b}^2)}}.
    \label{eq:low-frequency-rescaling-ratio}
\end{equation}
Every Fourier mode of $U$ becomes
\[
    e^{i(\pi mx/a+\pi ny/b)}
    =
    e^{2\pi i(mX+nY)}.
\]
If its eigenvalue satisfies $\lambda<\Lambda_*$, then
\[
    \frac{\pi^2m^2}{a^2}
    +
    \frac{\pi^2n^2}{b^2}
    <
    \Lambda_*,
\]
and therefore
\[
    |m|
    \leq
    N_x(\Lambda_*),
    \qquad
    |n|
    \leq
    N_y(\Lambda_*).
\]
The rescaled eigenfunction thus satisfies
\cref{lem:finite-bandwidth-observation}. Since the rescaling
preserves both the relative observation measure and the ratio of the
corresponding $L^2$ norms, the conclusion follows.
\end{proof}

\subsection{Full-Spectrum Quantitative Non-Localization}
\label{subsec:full-spectrum-quantitative}

The preceding estimates control complementary spectral regimes.
The high-frequency estimate gives an area-only lower bound above the
explicit threshold $\lambda_*(R,V)$, whereas the finite-bandwidth
estimate controls all eigenspaces below this threshold. Combining
the two gives an explicit quantitative lower bound over the entire
spectrum.


\begin{theorem}[Full-spectrum quantitative non-localization on rectangles]
\label{thm:full-spectrum-quantitative-rectangle}
Let
\[
    R=(0,a)\times(0,b),
\]
and let $V\subset R$ be measurable such that its fourfold reflected
extension $V^\sharp\subset\mathbb T_{a,b}^2$ has finite perimeter.
Assume
\[
    0<|V|<|R|,
\]
and set
\[
    \alpha=\frac{|V|}{|R|}.
\]
Let $\lambda_*(R,V)$ be defined by
\eqref{eq:lambda-star-explicit}, and set
\begin{equation}
    N_x^*
    :=
    \left\lfloor
        \frac{a\sqrt{\lambda_*(R,V)}}{\pi}
    \right\rfloor,
    \qquad
    N_y^*
    :=
    \left\lfloor
        \frac{b\sqrt{\lambda_*(R,V)}}{\pi}
    \right\rfloor.
\end{equation}
Define
\begin{equation}
    F_{\mathrm{LF}}(R,V)
    :=
    \left(
        \frac{\alpha}{2C_{\mathrm{NT}}}
    \right)^{2N_x^*+2N_y^*+1}
\end{equation}
and
\begin{equation}
    F_{\mathrm{HF}}(\alpha)
    :=
    \left[
        \frac{\alpha}{2}
        \left(
            1-
            \frac{\sin(\pi\alpha)}{\pi\alpha}
        \right)
    \right]^{1/2}.
\end{equation}
Finally, set
\begin{equation}
    F_R(V)
    :=
    \min
    \left\{
        F_{\mathrm{LF}}(R,V),
        F_{\mathrm{HF}}(\alpha)
    \right\}.
\end{equation}
Then every Dirichlet Laplacian
eigenfunction $u$ on $R$ satisfies
\begin{equation}
{
    \frac{\|u\|_{L^2(V)}}
         {\|u\|_{L^2(R)}}
    \geq
    F_R(V)>0.
    }
    \label{eq:full-spectrum-rectangle-final}
\end{equation}
Equivalently,
\begin{equation}
{
    C_2(V;R)
    \geq
    F_R(V)>0.
    }
    \label{eq:C2-full-spectrum-quantitative}
\end{equation}
The estimate is uniform in the eigenvalue, its multiplicity, and the
choice of eigenfunction within the corresponding eigenspace.
\end{theorem}

\begin{proof}
Let
\[
    -\Delta_Ru=\lambda u.
\]
If
\[
    \lambda\geq\lambda_*(R,V),
\]
then
Theorem \ref{thm:explicit-high-frequency-rectangle} gives
\[
    \frac{\|u\|_{L^2(V)}}
         {\|u\|_{L^2(R)}}
    \geq
    F_{\mathrm{HF}}(\alpha)
    \geq
    F_R(V).
\]
Suppose instead that
\[
    \lambda<\lambda_*(R,V).
\]
Let $\widetilde u$ be the odd periodic extension of $u$ to
$\mathbb T_{a,b}^2$, and let $V^\sharp$ be the fourfold reflected
extension of $V$. By
\eqref{eq:reflected-area-fraction},
\[
    \frac{|V^\sharp|}
         {|\mathbb T_{a,b}^2|}
    =
    \alpha.
\]
Applying
\cref{prop:explicit-low-frequency-torus} with
\[
    W=V^\sharp,
    \qquad
    \Lambda_*=\lambda_*(R,V),
\]
gives
\[
    \frac{
        \|\widetilde u\|_{L^2(V^\sharp)}
    }{
        \|\widetilde u\|_{L^2(\mathbb T_{a,b}^2)}
    }
    \geq
    F_{\mathrm{LF}}(R,V).
\]
Using the reflection identity
\eqref{eq:reflected-ratio-identity}, we obtain
\[
    \frac{\|u\|_{L^2(V)}}
         {\|u\|_{L^2(R)}}
    \geq
    F_{\mathrm{LF}}(R,V)
    \geq
    F_R(V).
\]
This proves the result.
\end{proof}


\begin{remark}[Two quantitative regimes]
\label{rem:two-quantitative-regimes}
The two constants entering
\cref{thm:full-spectrum-quantitative-rectangle} play different
roles. The low-frequency constant
\[
    F_{\mathrm{LF}}(R,V)
\]
is a finite-bandwidth completion and is not intended to be sharp.
Its dependence on the geometry of $V$ enters through the explicit
threshold $\lambda_*(R,V)$.

By contrast, above this threshold the lower bound becomes
\[
    F_{\mathrm{HF}}(\alpha)
    =
    \left[
        \frac{\alpha}{2}
        \left(
            1-
            \frac{\sin(\pi\alpha)}
                 {\pi\alpha}
        \right)
    \right]^{1/2},
\]
which depends only on the relative area
\[
    \alpha=\frac{|V|}{|R|}.
\]
Thus the quantitative theory separates two effects: the geometry of
the observation set determines when the high-frequency regime becomes
effective, whereas the non-localization constant within that regime
is universal among finite-perimeter observation sets having the same
area fraction.
\end{remark}

\section{Uniform $L^2$ Non-Localization on Rectangles and Equilateral Triangles}
\label{sec:uniform-L2-nonlocalization}
Let $\Omega\subset\mathbb{R}^2$ be a bounded domain, and let
$-\Delta_\Omega$ denote the Dirichlet Laplacian on $\Omega$.
Given a measurable set $V\subset\Omega$ with $|V|>0$, 
$C_2(V;\Omega)>0$ means that every Laplacian eigenfunction
carries a uniformly positive proportion of its $L^2$ mass in $V$,
uniformly with respect to both the eigenvalue and the choice of
eigenfunction inside a possibly degenerate eigenspace.

Nguyen and Grebenkov \cite{GrebenkovNguyen2013} established the
positivity of the analogous quantity for separated eigenfunctions on
rectangle-like domains with simple spectrum. They also observed that,
in the presence of spectral degeneracies, the corresponding problem
for arbitrary linear combinations is substantially more delicate; in
particular, the unit square was left open in their discussion.
We show below that the $L^2$ case follows from Schr\"odinger
observability.

\subsection{A General Observation}

In this section, we first present a simple consequence of Schr\"odinger observability.
\begin{lemma}[Observability implies uniform eigenfunction mass]
\label{lem:obs-to-eigen}
Let $\Omega$ be a bounded domain and let $V\subset\Omega$ be a measurable
subset. Assume that there exist $T>0$ and $C_{\mathrm{obs}}>0$ such that
the Schr\"odinger observability estimate
\begin{equation}
\|f\|_{L^2(\Omega)}^2
\leq
C_{\mathrm{obs}}
\int_0^T
\int_V
\left|e^{it\Delta_\Omega}f(x)\right|^2
\,dx\,dt
\label{eq:obs-general}
\end{equation}
holds for every $f\in L^2(\Omega)$.
Then, every Laplacian eigenfunction $u\in E_\lambda(\Omega)$ satisfies
\begin{equation}
\|u\|_{L^2(V)}
\geq
(TC_{\mathrm{obs}})^{-1/2}
\|u\|_{L^2(\Omega)}.
\label{eq:eigen-mass-general}
\end{equation}
In particular,
\[
C_2(V;\Omega)>0.
\]
\end{lemma}

\begin{proof}
Let
\[
    u\in E_\lambda(\Omega)\setminus\{0\}.
\]
Since
\[
    -\Delta_\Omega u=\lambda u,
\]
the Schr\"odinger evolution preserves the spatial profile of $u$ up
to a unimodular phase:
\begin{equation}
    e^{it\Delta_\Omega}u
    =
    e^{-it\lambda}u.
    \label{eq:eigen-phase}
\end{equation}
Hence
\[
    \left|e^{it\Delta_\Omega}u(x)\right|
    =
    |u(x)|
\]
for all $t$ and almost every $x\in\Omega$.
Applying the observability estimate in the inequality \eqref{eq:obs-general} with
$f=u$ therefore gives
\begin{align}
    \|u\|_{L^2(\Omega)}^2
    &\leq
    C_{\mathrm{obs}}
    \int_0^T\int_V
    |u(x)|^2\,dx\,dt
    \nonumber\\
    &=
    T C_{\mathrm{obs}}
    \|u\|_{L^2(V)}^2.
\end{align}
Therefore,
\begin{equation}
    \|u\|_{L^2(V)}
    \geq
    (T C_{\mathrm{obs}})^{-1/2}
    \|u\|_{L^2(\Omega)},
\end{equation}
which proves the inequality \eqref{eq:eigen-mass-general}.

Since the constant $(T C_{\mathrm{obs}})^{-1/2}$ is independent of
$\lambda$, its multiplicity, and the choice of
$u\in E_\lambda(\Omega)$, taking the infimum over all eigenvalues and
all nonzero eigenfunctions yields
\[
    C_2(V;\Omega)
    \geq
    (T C_{\mathrm{obs}})^{-1/2}
    >0.
\]
\end{proof}
\subsection{Uniform $L^2$ Non-Localization on Rectangles}
Let
\[
R=(0,a)\times(0,b),
\qquad a,b>0.
\]
The standard Dirichlet eigenfunctions are
\begin{equation}
\phi_{m,n}(x,y)
=
\sin\left(\frac{m\pi x}{a}\right)
\sin\left(\frac{n\pi y}{b}\right),
\qquad
m,n\in\mathbb{N},
\label{eq:rectangle-eigenfunctions}
\end{equation}
with the corresponding eigenvalues
\begin{equation}
\lambda_{m,n}
=
\pi^2
\left(
\frac{m^2}{a^2}
+
\frac{n^2}{b^2}
\right).
\label{eq:rectangle-eigenvalues}
\end{equation}
When $\lambda$ is degenerate, an arbitrary element of
$E_\lambda(R)$ is a linear combination of different modes
in Eq. \eqref{eq:rectangle-eigenfunctions} having eigenvalue $\lambda$.

\begin{theorem}[Uniform $L^2$ non-localization on rectangles]
\label{thm:rectangle-nonlocalization}
Let $R=(0,a)\times(0,b)$ and let $V\subset R$ be measurable with
$|V|>0$. Then, there exists a constant $c_{R,V}>0$ such that
\begin{equation}
\|u\|_{L^2(V)}
\geq
c_{R,V}\|u\|_{L^2(R)}
\label{eq:rectangle-main-bound}
\end{equation}
for every Dirichlet Laplacian eigenfunction $u$ on $R$.
The constant $c_{R,V}$ is independent of the eigenvalue, its
multiplicity, and the choice of $u$ in the corresponding eigenspace.
Consequently,
\[
C_2(V;R)>0.
\]
\end{theorem}

\begin{proof}
Let
\[
    \mathbb{T}_{a,b}^2
    :=
    \mathbb{R}^2/
    \left(2a\mathbb{Z}\times2b\mathbb{Z}\right)
\]
be the rectangular flat torus associated with
$R=(0,a)\times(0,b)$. Let us fix
\[
    \lambda\in\sigma(-\Delta_R),
    \qquad
    u\in E_\lambda(R)\setminus\{0\}.
\]
Now, we extend $u$ oddly across the sides of $R$ and then periodically with
periods $2a$ and $2b$, and denote the resulting function on
$\mathbb{T}_{a,b}^2$ by $\widetilde u$. Since every element of
$E_\lambda(R)$ is a finite linear combination of the Dirichlet modes
\[
    \sin\left(\frac{m\pi x}{a}\right)
    \sin\left(\frac{n\pi y}{b}\right)
\]
having the eigenvalue $\lambda$, the extension $\widetilde u$ is a
Laplacian eigenfunction on $\mathbb{T}_{a,b}^2$ with the same
eigenvalue:
\begin{equation}
    -\Delta_{\mathbb{T}_{a,b}^2}\widetilde u
    =
    \lambda\widetilde u.
    \label{eq:extended-eigen}
\end{equation}
Moreover, the fundamental domain
$(-a,a)\times(-b,b)$ consists of four reflected copies of $R$.
Therefore,
\begin{equation}
    \|\widetilde u\|_{L^2(\mathbb{T}_{a,b}^2)}^2
    =
    4\|u\|_{L^2(R)}^2.
    \label{eq:reflection-norm}
\end{equation}
Since $\widetilde u=u$ on the original copy of $R$, one also has
\begin{equation}
    \|\widetilde u\|_{L^2(V)}
    =
    \|u\|_{L^2(V)}.
    \label{eq:rect-local-norm}
\end{equation}
The torus $\mathbb{T}_{a,b}^2$ is a rectangular flat torus of the
type covered by the observability result of Burq and Zworski.
Indeed, after a uniform spatial dilation to the normalization used
in~\cite[Theorem~1.3 and Remark~1]{BurqZworski2019}, together with
the corresponding rescaling of the time variable in the free
Schr\"odinger equation, their observability estimate transfers to
$\mathbb{T}_{a,b}^2$, with a possibly different observability
constant.

Now, let us fix an observation time $\tau_0>0$. Since
$V\subset\mathbb{T}_{a,b}^2$ is measurable and $|V|>0$, the
Burq--Zworski observability result yields a constant
$C_{V,\tau_0}>0$ such that
\begin{equation}
    \|F\|_{L^2(\mathbb{T}_{a,b}^2)}^2
    \leq
    C_{V,\tau_0}
    \int_0^{\tau_0}
    \int_V
    \left|
        e^{it\Delta_{\mathbb{T}_{a,b}^2}}F(x)
    \right|^2
    \,dx\,dt
    \label{eq:BZ-observability}
\end{equation}
for every $F\in L^2(\mathbb{T}_{a,b}^2)$.

Applying Lemma~\ref{lem:obs-to-eigen} on
$\mathbb{T}_{a,b}^2$ to the eigenfunction $\widetilde u$,
\[
    \|\widetilde u\|_{L^2(V)}
    \geq
    \frac{1}{\sqrt{\tau_0 C_{V,\tau_0}}}
    \|\widetilde u\|_{L^2(\mathbb{T}_{a,b}^2)}.
\]
Using Eqs.~\eqref{eq:reflection-norm} and
\eqref{eq:rect-local-norm}, we obtain
\begin{equation}
    \|u\|_{L^2(V)}
    \geq
    \frac{2}{\sqrt{\tau_0 C_{V,\tau_0}}}
    \|u\|_{L^2(R)}.
    \label{eq:rectangle-constant}
\end{equation}
Thus, the inequality ~\eqref{eq:rectangle-main-bound} holds with
\[
    c_{R,V}
    :=
    \frac{2}{\sqrt{\tau_0 C_{V,\tau_0}}}>0.
\]
For the fixed choice of $\tau_0$, the observability constant
$C_{V,\tau_0}$ depends only on the observation set $V$ and the fixed
rectangular torus $\mathbb{T}_{a,b}^2$. Hence, $c_{R,V}$ is independent
of the eigenvalue $\lambda$, its multiplicity, and the choice of
$u\in E_\lambda(R)$. Taking the infimum over all eigenvalues and all
nonzero eigenfunctions yields
\[
    C_2(V;R)
    \geq
    c_{R,V}>0,
\]
which completes the proof.
\end{proof}

\begin{remark}
\label{rem:rectangle-degeneracy}
The essential point in Theorem~\ref{thm:rectangle-nonlocalization} is
that no simplicity assumption on the spectrum is needed. The proof
does not estimate the coefficients of an eigenfunction in a separated
basis. Instead, the entire eigenspace is transported to a toral
eigenspace and controlled by an observability estimate that is uniform
over all initial data.
\end{remark}

\subsection{Uniform $L^2$ Non-Localization on Equilateral Triangles}

Let $\mathcal{T}\subset\mathbb{R}^2$ be an equilateral triangle, and
let $-\Delta_{\mathcal{T}}$ denote its Dirichlet Laplacian.
The spectral decomposition of $-\Delta_{\mathcal{T}}$ is explicitly
known through finite trigonometric sums; see
Pinsky \cite{Pinsky1985} and the references therein.

However, for the results mentioned
below, there is no explicit formula needed for the eigenfunctions.
\begin{theorem}[Uniform $L^2$ non-localization on an equilateral triangle]
\label{thm:triangle-nonlocalization}
Let $\mathcal{T}$ be an equilateral triangle and let
$V\subset\mathcal{T}$ be measurable with $|V|>0$. Then, there exists
$c_{\mathcal{T},V}>0$ such that
\begin{equation}
\|u\|_{L^2(V)}
\geq
c_{\mathcal{T},V}
\|u\|_{L^2(\mathcal{T})},
\label{eq:triangle-main-bound}
\end{equation}
for every Dirichlet Laplacian eigenfunction $u$ on $\mathcal{T}$.
The constant is uniform with respect to the eigenvalue, its
multiplicity, and the choice of eigenfunction. Consequently,
\[
C_2(V;\mathcal{T})>0.
\]
\end{theorem}

\begin{proof}
Let $\mathcal{T}_0$ denote the normalized equilateral triangle used
in Alphonse and Lafontaine~\cite[Theorem~1.2]
{AlphonseLafontaine2025}. Since every equilateral triangle is similar
to $\mathcal{T}_0$, there exist $s>0$, an orthogonal transformation
$R$, and $x_0\in\mathbb{R}^2$ such that
\[
    \mathcal{T}
    =
    \Phi(\mathcal{T}_0),
    \qquad
    \Phi(y)=x_0+sRy.
\]
Define
\[
    U:L^2(\mathcal{T})\longrightarrow L^2(\mathcal{T}_0),
    \qquad
    (Uf)(y)=s\,f(\Phi(y)).
\]
Since the spatial dimension is two, the factor $s$ makes $U$ unitary.
Moreover,
\[
    U(-\Delta_{\mathcal{T}})U^{-1}
    =
    s^{-2}(-\Delta_{\mathcal{T}_0}),
\]
and hence
\begin{equation}
    Ue^{it\Delta_{\mathcal{T}}}U^{-1}
    =
    e^{i(t/s^2)\Delta_{\mathcal{T}_0}}.
    \label{eq:triangle-scaling-flow}
\end{equation}

Let
\[
    V_0:=\Phi^{-1}(V)\subset\mathcal{T}_0.
\]
Since $|V|>0$, we also have $|V_0|>0$. Now, let us fix an observation time
$\tau_0>0$. By the Schr\"odinger observability theorem of Alphonse
and Lafontaine~\cite[Theorem~1.2]{AlphonseLafontaine2025}, applied
to $\mathcal{T}_0$ with observation set $V_0$ and observation time
$\tau_0/s^2$, there exists a constant
$C_{V_0,\tau_0/s^2}>0$ such that
\begin{equation}
    \|g\|_{L^2(\mathcal{T}_0)}^2
    \leq
    C_{V_0,\tau_0/s^2}
    \int_0^{\tau_0/s^2}
    \int_{V_0}
    \left|
        e^{ir\Delta_{\mathcal{T}_0}}g(y)
    \right|^2
    \,dy\,dr
    \label{eq:triangle-observability-normalized}
\end{equation}
for every $g\in L^2(\mathcal{T}_0)$.

Taking $g=Uf$ in
the inequality~\eqref{eq:triangle-observability-normalized}, using the unitarity
of $U$ and Eq.~\eqref{eq:triangle-scaling-flow}, and then making the
change of variables
\[
    t=s^2r,
\]
one has
\begin{equation}
    \|f\|_{L^2(\mathcal{T})}^2
    \leq
    C_{V,\tau_0}
    \int_0^{\tau_0}
    \int_V
    \left|
        e^{it\Delta_{\mathcal{T}}}f(x)
    \right|^2
    \,dx\,dt,
    \label{eq:triangle-observability-V}
\end{equation}
where
\begin{equation}
    C_{V,\tau_0}
    :=
    s^{-2}C_{V_0,\tau_0/s^2}>0.
    \label{eq:triangle-scaled-observability-constant}
\end{equation}
Now, let
\[
    u\in E_\lambda(\mathcal{T})\setminus\{0\}.
\]
Applying Lemma~\ref{lem:obs-to-eigen} to
the inequality~\eqref{eq:triangle-observability-V} gives
\[
    \|u\|_{L^2(V)}
    \geq
    \frac{1}{\sqrt{\tau_0 C_{V,\tau_0}}}
    \|u\|_{L^2(\mathcal{T})}.
\]
Thus, the inequality~\eqref{eq:triangle-main-bound} holds with
\[
    c_{\mathcal{T},V}
    :=
    \frac{1}{\sqrt{\tau_0 C_{V,\tau_0}}}>0.
\]
For the fixed choice of $\tau_0$, this constant depends only on the
fixed triangle $\mathcal{T}$ and the observation set $V$, and is
independent of the eigenvalue $\lambda$, its multiplicity, and the
choice of $u$ within $E_\lambda(\mathcal{T})$. Consequently,
\[
    C_2(V;\mathcal{T})
    \geq
    c_{\mathcal{T},V}>0.
\]

\end{proof}
\begin{remark}
\label{rem:positive-measure}
Theorems~\ref{thm:rectangle-nonlocalization} and
\ref{thm:triangle-nonlocalization} apply to arbitrary measurable sets
of positive Lebesgue measure. This observation class is broader than
the class of nonempty open subsets considered in the rectangle-like
non-localization criterion of Nguyen and
Grebenkov~\cite[Theorem~4.1]{GrebenkovNguyen2013}.
\end{remark}

\section{Uniform $L^2$ Non-Localization on Integrable Polygons}
\label{sec:integrable-polygons}

Throughout this paper, by an integrable polygon we mean a polygonal
domain belonging to the trigonometric spectral classification of
McCartin. More precisely, McCartin's classification
\cite[Theorem~2]{McCartin2008} shows that, up to similarity, the
planar polygonal domains admitting a complete trigonometric system
of Laplacian eigenfunctions are precisely the following four classes:
rectangles, isosceles right triangles, equilateral triangles, and
hemi-equilateral triangles, the latter being
$30^\circ$--$60^\circ$--$90^\circ$ triangles.

The observability estimates needed below are available from existing
Schr\"odinger observability theory. Burq and Zworski provide the
relevant estimate in a rectangular-torus setting, while Alphonse
and Lafontaine establish the corresponding result for the equilateral
triangle. They further observe that a slight modification of their
argument gives analogous observability results for the four integrable
polygonal classes listed above; see
\cite[Section~1.3]{AlphonseLafontaine2025}.

We use these observability results to obtain stationary estimates for
Dirichlet Laplacian eigenspaces. In particular, for every measurable
set $V\subset\Omega$ of positive measure, the resulting estimate gives
a positive lower bound for the $L^2$ mass on $V$ that is uniform over
each eigenspace, including eigenspaces of nontrivial multiplicity.
Equivalently, in terms of the non-localization quantity introduced
above,
\[
    C_2(V;\Omega)>0.
\]
The symmetry reductions used below allow us to pass between the
relevant polygonal geometries while keeping track of the corresponding
$L^2$ norm factors. This stationary estimate also provides the
exact-eigenfunction case for the quasimode and spectral-cluster
estimates developed later.

Let $\Omega$ be one of the domains in this class of integrable polygons, 
and let $-\Delta_\Omega$ denote the Dirichlet Laplacian on $\Omega$.
We will prove the following results.
\begin{lemma}[Odd reflection across a Dirichlet side]
\label{lem:odd-reflection}
Let $D\subset\mathbb{R}^2$ be a bounded polygonal domain and let
$S\subset\partial D$ be a straight side. Let $\rho$ be the Euclidean
reflection across the line containing $S$, and assume that
$D$ and $\rho(D)$ have disjoint interiors and meet along $S$. Set
\[
    D^\ast
    :=
    \operatorname{int}
    \bigl(\overline{D}\cup\overline{\rho(D)}\bigr).
\]
Suppose that $u\in H_0^1(D)$ satisfies
\begin{equation}
    \int_D \nabla u\cdot\nabla\overline{\varphi}\,dx
    =
    \lambda
    \int_D u\,\overline{\varphi}\,dx
    \qquad
    \text{for every }\varphi\in H_0^1(D).
    \label{eq:weak-eigen-D}
\end{equation}
Define the odd extension
\[
    \widetilde u(x)
    :=
    \begin{cases}
        u(x), & x\in D,\\[1mm]
        -u(\rho x), & x\in\rho(D).
    \end{cases}
\]
Then, $\widetilde u\in H_0^1(D^\ast)$ and
\begin{equation}
    \int_{D^\ast}
    \nabla\widetilde u\cdot\nabla\overline{\psi}\,dx
    =
    \lambda
    \int_{D^\ast}
    \widetilde u\,\overline{\psi}\,dx
    \qquad
    \text{for every }\psi\in H_0^1(D^\ast).
    \label{eq:weak-eigen-reflected}
\end{equation}
Consequently,
\[
    \widetilde u\in E_\lambda(D^\ast).
\]
Moreover,
\[
    \|\widetilde u\|_{L^2(D^\ast)}^2
    =
    2\|u\|_{L^2(D)}^2.
\]
\end{lemma}

\begin{proof}
Since $u$ has zero trace on $S$, its odd reflection across $S$
has matching traces on the two sides of the interface. 
Now, let us choose $u_n\in C_c^\infty(D)$ with $u_n\to u$ in $H^1(D)$, and let $\widetilde u_n$ be the odd reflection of $u_n$ across $S$.
Since $u_n$ vanishes in a neighborhood of $\partial D$, one has $\widetilde u_n\in C_c^\infty(D^\ast)$.
Reflection is an isometry, so
\[
    \|\widetilde u_n-\widetilde u_m\|_{H^1(D^\ast)}^2
    =
    2\|u_n-u_m\|_{H^1(D)}^2.
\]
Hence $\widetilde u_n$ converges in $H^1(D^\ast)$ to the function $\widetilde u$, and therefore $\widetilde u\in H_0^1(D^\ast)$. 
Let $\psi\in H_0^1(D^\ast)$ and define, on $D$,
\[
    \varphi
    :=
    \psi|_D-(\psi\circ\rho)|_D.
\]
The traces of the two terms agree on $S$, while both vanish on the
corresponding exterior boundary portions; hence
$\varphi\in H_0^1(D)$. Using the change of variables $x=\rho y$ on
the reflected copy and the fact that $\rho$ is an orthogonal
isometry, we obtain
\begin{align*}
    \int_{D^\ast}
    \nabla\widetilde u\cdot\nabla\overline{\psi}\,dx
    &=
    \int_D
    \nabla u\cdot
    \nabla\overline{\bigl(\psi-\psi\circ\rho\bigr)}\,dx\\
    &=
    \lambda
    \int_D
    u\,\overline{\bigl(\psi-\psi\circ\rho\bigr)}\,dx\\
    &=
    \lambda
    \int_{D^\ast}
    \widetilde u\,\overline{\psi}\,dx.
\end{align*}
Thus, $\widetilde u$ satisfies the Dirichlet eigenvalue equation on
$D^\ast$ in the weak sense. By the operator definition of
$-\Delta_{D^\ast}$ given above, it follows that
\[
    \widetilde u\in E_\lambda(D^\ast).
\]
Finally, since $\rho$ is measure preserving,
\[
    \|\widetilde u\|_{L^2(D^\ast)}^2
    =
    \|u\|_{L^2(D)}^2
    +
    \|u\circ\rho\|_{L^2(\rho(D))}^2
    =
    2\|u\|_{L^2(D)}^2.
\]
\end{proof}
We now state the corresponding stationary estimate for Laplacian
eigenspaces. The underlying Schr\"odinger observability estimates are
known; the point here is that they yield a lower bound for
$C_2(V;\Omega)$ that is uniform over each eigenspace, including in the
presence of spectral multiplicity.

\begin{theorem}[Uniform $L^2$ non-localization on integrable polygons]
\label{thm:integrable-polygons}
Let $\Omega\subset\mathbb{R}^2$ be an integrable polygon, i.e., one of
the following:
\begin{enumerate}[label=(\roman*)]
    \item a rectangle;
    \item an isosceles right triangle;
    \item an equilateral triangle;
    \item a hemi-equilateral $(30^\circ$--$60^\circ$--$90^\circ)$ triangle.
\end{enumerate}
Let $V\subset\Omega$ be measurable with positive Lebesgue measure,
$|V|>0$.
Then, there exists a constant
\begin{equation}
    c_{\Omega,V}>0
\end{equation}
such that every Dirichlet Laplacian eigenfunction
$u\in E_\lambda(\Omega)$ satisfies
\begin{equation}
{
    \|u\|_{L^2(V)}
    \geq
    c_{\Omega,V}
    \|u\|_{L^2(\Omega)}.
    }
    \label{eq:integrable-main-bound}
\end{equation}
The constant $c_{\Omega,V}$ is independent of the eigenvalue
$\lambda$, its multiplicity, and the particular choice of
$u\in E_\lambda(\Omega)$.
Consequently,
\begin{equation}
{
    C_2(V;\Omega)>0.
    }
    \label{eq:C2-integrable-positive}
\end{equation}
\end{theorem}

\begin{figure}[htbp]
    \centering
    \begin{minipage}[t]{0.46\textwidth}
        \centering
        \begin{tikzpicture}[scale=1.0, line width=0.9pt]
            \coordinate (A) at (0,0);
            \coordinate (B) at (4,0);
            \coordinate (C) at (4,2.4);
            \coordinate (D) at (0,2.4);

            \draw (A)--(B)--(C)--(D)--cycle;

            \draw (0.28,0)--(0.28,0.28)--(0,0.28);

            \node at (2,1.2) {$\Omega$};

            \node[below] at (2,0) {$a$};
            \node[right] at (4,1.2) {$b$};
        \end{tikzpicture}

        \vspace{1mm}
        {\small (a) Rectangle}
    \end{minipage}
    \hfill
    \begin{minipage}[t]{0.46\textwidth}
        \centering
        \begin{tikzpicture}[scale=1.0, line width=0.9pt]
            \coordinate (A) at (0,0);
            \coordinate (B) at (3.4,0);
            \coordinate (C) at (0,3.4);

            \draw (A)--(B)--(C)--cycle;

            \draw (0.28,0)--(0.28,0.28)--(0,0.28);

            \draw (1.65,-0.09)--(1.65,0.09);
            \draw (-0.09,1.65)--(0.09,1.65);

            \node at (0.95,0.95) {$\Omega$};

            \node at (2.72,0.28) {$45^\circ$};
            \node at (0.37,2.72) {$45^\circ$};
        \end{tikzpicture}

        \vspace{1mm}
        {\small (b) Isosceles right triangle}
    \end{minipage}

    \vspace{5mm}

    \begin{minipage}[t]{0.46\textwidth}
        \centering
        \begin{tikzpicture}[scale=1.0, line width=0.9pt]
            \coordinate (A) at (0,0);
            \coordinate (B) at (4,0);
            \coordinate (C) at (2,3.464);

            \draw (A)--(B)--(C)--cycle;

            \draw (1.95,-0.09)--(1.95,0.09);

            \draw
                ($(A)!0.5!(C)+(-0.08,0.05)$)
                --
                ($(A)!0.5!(C)+(0.08,-0.05)$);

            \draw
                ($(B)!0.5!(C)+(-0.08,-0.05)$)
                --
                ($(B)!0.5!(C)+(0.08,0.05)$);

            \node at (2,1.25) {$\Omega$};

            \node at (0.52,0.28) {$60^\circ$};
            \node at (3.48,0.28) {$60^\circ$};
            \node at (2,2.92) {$60^\circ$};
        \end{tikzpicture}

        \vspace{1mm}
        {\small (c) Equilateral triangle}
    \end{minipage}
    \hfill
    \begin{minipage}[t]{0.46\textwidth}
        \centering
        \begin{tikzpicture}[scale=1.0, line width=0.9pt]
            \coordinate (A) at (0,0);
            \coordinate (B) at (2,0);
            \coordinate (C) at (2,3.464);

            \draw (A)--(B)--(C)--cycle;

            \draw (1.72,0)--(1.72,0.28)--(2,0.28);

            \node at (1.35,1.12) {$\Omega$};

            \node at (0.48,0.28) {$60^\circ$};
            \node at (1.72,2.82) {$30^\circ$};
            \node at (1.66,0.42) {$90^\circ$};

            \draw[dashed, line width=0.6pt]
                (2,0)--(2,3.464);
        \end{tikzpicture}

        \vspace{1mm}
        {\small (d) Hemi-equilateral triangle}
    \end{minipage}

    \caption{
    The four classes of integrable polygonal domains considered in
    Theorem~\ref{thm:integrable-polygons}:
    (a) a rectangle,
    (b) an isosceles right triangle,
    (c) an equilateral triangle, and
    (d) a hemi-equilateral
    ($30^\circ$--$60^\circ$--$90^\circ$) triangle.
    }
    \label{fig:integrable-polygons}
\end{figure}

\begin{proof}
The rectangle and equilateral-triangle cases follow directly from
Theorems~\ref{thm:rectangle-nonlocalization} and
\ref{thm:triangle-nonlocalization}, respectively. It therefore remains
to consider the two integrable triangular domains obtained by symmetry
reduction.

\medskip
\noindent
\textbf{Isosceles right triangle.}
Let $T_{\mathrm{iso}}$ be an isosceles right triangle and let $Q$ be
the square obtained by reflecting $T_{\mathrm{iso}}$ across its
hypotenuse. For
\[
    u\in E_\lambda(T_{\mathrm{iso}})\setminus\{0\},
\]
let $\widetilde u$ denote its odd reflection across the hypotenuse.
By Lemma~\ref{lem:odd-reflection},
\[
    \widetilde u\in E_\lambda(Q),
\]
and
\begin{equation}
    \|\widetilde u\|_{L^2(Q)}^2
    =
    2\|u\|_{L^2(T_{\mathrm{iso}})}^2.
    \label{eq:proof-isosceles-norm}
\end{equation}
For a measurable set $V\subset T_{\mathrm{iso}}$ with $|V|>0$,
Theorem~\ref{thm:rectangle-nonlocalization}, applied to $Q$, gives
\[
    \|\widetilde u\|_{L^2(V)}
    \geq
    c_{Q,V}\|\widetilde u\|_{L^2(Q)}.
\]
Since $\widetilde u=u$ on $V$, we obtain
\[
    \|u\|_{L^2(V)}
    \geq
    \sqrt{2}\,c_{Q,V}
    \|u\|_{L^2(T_{\mathrm{iso}})}.
\]

\medskip
\noindent

\textbf{Hemi-equilateral triangle.}
Let $T_{\mathrm{hemi}}$ be a hemi-equilateral triangle and let
$T_{\mathrm{eq}}$ be the equilateral triangle obtained by reflecting
$T_{\mathrm{hemi}}$ across its altitude. For
\[
    u\in E_\lambda(T_{\mathrm{hemi}})\setminus\{0\},
\]
let $\widetilde u$ denote its odd reflection across the altitude.
By Lemma~\ref{lem:odd-reflection},
\[
    \widetilde u\in E_\lambda(T_{\mathrm{eq}}),
\]
and
\begin{equation}
    \|\widetilde u\|_{L^2(T_{\mathrm{eq}})}^2
    =
    2\|u\|_{L^2(T_{\mathrm{hemi}})}^2.
    \label{eq:proof-hemi-norm}
\end{equation}
Applying Theorem~\ref{thm:triangle-nonlocalization} to
$T_{\mathrm{eq}}$ with the same observation set
$V\subset T_{\mathrm{hemi}}\subset T_{\mathrm{eq}}$ gives
\[
    \|\widetilde u\|_{L^2(V)}
    \geq
    c_{T_{\mathrm{eq}},V}
    \|\widetilde u\|_{L^2(T_{\mathrm{eq}})}.
\]
Since $\widetilde u=u$ on $T_{\mathrm{hemi}}$, and in particular on $V$,
Eq.~\eqref{eq:proof-hemi-norm} yields
\begin{align}
    \|u\|_{L^2(V)}
    &=
    \|\widetilde u\|_{L^2(V)}
    \notag\\
    &\geq
    c_{T_{\mathrm{eq}},V}
    \|\widetilde u\|_{L^2(T_{\mathrm{eq}})}
    \notag\\
    &=
    \sqrt{2}\,c_{T_{\mathrm{eq}},V}
    \|u\|_{L^2(T_{\mathrm{hemi}})}.
    \label{eq:proof-hemi-final}
\end{align}
As $c_{T_{\mathrm{eq}},V}$ depends only on the fixed equilateral
triangle $T_{\mathrm{eq}}$ and the observation set $V$, the constant
$\sqrt{2}\,c_{T_{\mathrm{eq}},V}$ is independent of $\lambda$, its
multiplicity, and the choice of
$u\in E_\lambda(T_{\mathrm{hemi}})$.

\end{proof}

\begin{remark}[Uniformity and interpretation of the non-localization bound]
\label{rem:integrable-uniformity}
The strength of Theorem~\ref{thm:integrable-polygons} lies in the
uniformity of the constant $c_{\Omega,V}$. For a fixed measurable set
$V\subset\Omega$ with $|V|>0$, the same constant applies
simultaneously to every eigenvalue and to every eigenfunction in the
corresponding eigenspace. Equivalently,
\begin{equation}
    \inf_{\lambda\in\sigma(-\Delta_\Omega)}
    \;
    \inf_{\substack{u\in E_\lambda(\Omega)\setminus\{0\}}}
    \frac{\|u\|_{L^2(V)}}
         {\|u\|_{L^2(\Omega)}}
    >0.
\end{equation}
Thus, the assertion is substantially stronger than the qualitative
statement that an individual eigenfunction cannot vanish identically
on a set of positive measure: it provides a quantitative lower bound
that is uniform over the entire spectrum.

In particular, spectral multiplicity does not create an exception to
the estimate. If $\lambda$ is a multiple eigenvalue and
$\{\phi_1,\ldots,\phi_{d_\lambda}\}$ is any basis of
$E_\lambda(\Omega)$, then the estimate holds uniformly for every
linear combination
\[
    u=\sum_{k=1}^{d_\lambda}c_k\phi_k.
\]
Hence, no choice of coefficients within a degenerate eigenspace can
produce a normalized eigenfunction whose $L^2$ mass on $V$ becomes
arbitrarily small.

Another notable feature is that $V$ is required only to be measurable
and to have positive Lebesgue measure; no openness, smoothness, or
geometric control condition is imposed on the observation set.
Consequently, for every sequence $\{u_j\}$ of normalized Dirichlet
eigenfunctions,
\[
    \|u_j\|_{L^2(\Omega)}=1,
\]
one has
\begin{equation}
    \inf_{j\geq1}\|u_j\|_{L^2(V)}
    \geq c_{\Omega,V}>0.
\end{equation}
In particular, no such sequence can asymptotically avoid a fixed
positive-measure subset of an integrable polygon.

We emphasize, however, that this conclusion is a uniform
non-localization, or non-avoidance, property rather than an
equidistribution statement. The theorem does not prescribe how the
remaining $L^2$ mass is distributed throughout $\Omega$, nor does it
assert convergence of $|u_j|^2\,dx$ to normalized Lebesgue measure in
the high-frequency limit. Its content is precisely that every fixed
positive-measure region retains a uniformly positive amount of
eigenfunction mass over the entire spectrum.
\end{remark}

\subsection{Geometric Interpretation of the Integrable Cases}
\label{subsec:integrable-geometric-interpretation}

The four integrable polygonal classes appearing in
Theorem~\ref{thm:integrable-polygons} can be organized around two
basic geometric settings. Rectangles and isosceles right triangles
are linked to rectangular flat-torus geometry, whereas equilateral
and hemi-equilateral triangles are linked to the geometry of the
equilateral triangle and its associated triangular tiling.

For a rectangle, odd reflection across the sides followed by periodic
extension produces an eigenfunction on a rectangular flat torus, to
which the observability result of Burq and Zworski applies. An
isosceles right triangle reduces to the same setting after one
additional odd reflection across its hypotenuse, which produces a
Dirichlet eigenfunction on the associated square. Thus these two
classes ultimately inherit their non-localization estimates from
rectangular-torus observability.

The equilateral triangle provides the second parent geometry. Its
reflection structure generates the triangular tiling underlying the
Schr\"odinger observability result of Alphonse and Lafontaine. A
hemi-equilateral triangle reduces to this case by odd reflection
across its altitude, which extends a Dirichlet eigenfunction to an
antisymmetric Dirichlet eigenfunction on the full equilateral
triangle. Hence the remaining two classes inherit their estimates
from equilateral-triangle observability.

Accordingly, the geometric reductions themselves do not constitute
four independent observability arguments. Their role is to place the
Dirichlet spectral problems into geometries for which the required
Schr\"odinger observability estimates are already available. The
analytical passage from observability to a uniform eigenfunction
estimate is then precisely Lemma~\ref{lem:obs-to-eigen}, while the
finite norm changes introduced by odd reflection are accounted for
in the proof of Theorem~\ref{thm:integrable-polygons}.

\subsection{Consequences of Uniform Non-Localization of Eigenfunctions}
\label{subsec:consequences-localization}
\begin{corollary}[Uniform non-avoidance of positive-measure sets]
\label{cor:no-avoidance}
Let $\Omega\subset\mathbb{R}^2$ be an integrable polygon; that is, let
$\Omega$ be a rectangle, an isosceles right triangle, an equilateral
triangle, or a hemi-equilateral triangle. Let
$V\subset\Omega$ be measurable with $|V|>0$.\\
Then, there exists a constant $c_{\Omega,V}>0$ such that, for every
sequence $\{u_j\}_{j\geq1}$ of Dirichlet Laplacian eigenfunctions
satisfying
\begin{equation}
    -\Delta_\Omega u_j=\lambda_j u_j,
    \qquad
    \|u_j\|_{L^2(\Omega)}=1,
    \label{eq:normalized-eigenfunctions}
\end{equation}
one has
\begin{equation}
    \|u_j\|_{L^2(V)}
    \geq
    c_{\Omega,V}
    \qquad
    \text{for every }j\geq1.
    \label{eq:no-avoidance}
\end{equation}
Equivalently,
\begin{equation}
    \inf_{j\geq1}
    \|u_j\|_{L^2(V)}
    \geq
    c_{\Omega,V}>0.
    \label{eq:uniform-no-avoidance}
\end{equation}
In particular,
\begin{equation}
    \|u_j\|_{L^2(V)}
    \not\longrightarrow0
    \qquad
    \text{as }j\to\infty.
    \label{eq:no-mass-vanishing}
\end{equation}
\end{corollary}
The uniform lower bounds established above yield a direct
non-localization consequence: no sequence of normalized Laplacian eigenfunctions can asymptotically avoid a fixed measurable set of
positive measure. More precisely, the $L^2$ mass on any such set
remains uniformly bounded away from zero along every normalized
eigenfunction sequence.
\begin{proof}
By Theorem~\ref{thm:integrable-polygons}, there exists a constant
$c_{\Omega,V}>0$, depending only on the domain $\Omega$ and the
observation set $V$, such that every Dirichlet Laplacian eigenfunction
$u$ on $\Omega$ satisfies
\begin{equation}
    \|u\|_{L^2(V)}
    \geq
    c_{\Omega,V}
    \|u\|_{L^2(\Omega)}.
    \label{eq:corollary-uniform-bound}
\end{equation}
Applying this estimate to $u=u_j$ and using the normalization
$\|u_j\|_{L^2(\Omega)}=1$ yields
\begin{equation}
    \|u_j\|_{L^2(V)}
    \geq
    c_{\Omega,V}
\end{equation}
for every $j\geq1$. Since the same constant applies to the entire
sequence, the inequality \eqref{eq:uniform-no-avoidance} follows immediately.
In particular, the sequence
$\{\|u_j\|_{L^2(V)}\}_{j\geq1}$ is bounded away from zero and therefore
cannot converge to zero.
\end{proof}

Corollary~\ref{cor:no-avoidance} provides a direct spectral
interpretation of the observability estimates used above. For every
integrable polygon $\Omega$ and every measurable subset
$V\subset\Omega$ of positive measure, one has $C_2(V;\Omega)>0$.
Thus, every positive-measure subset of $\Omega$ captures a uniformly
positive fraction of the $L^2$ mass of every Dirichlet Laplacian
eigenfunction. The lower bound is uniform over the entire spectrum,
including eigenvalues of arbitrary finite multiplicity and arbitrary
linear combinations within the corresponding eigenspaces.

This conclusion rules out a specific form of eigenfunction
localization: there cannot exist a sequence of $L^2$-normalized
Dirichlet eigenfunctions whose mass asymptotically vanishes on a fixed
positive-measure subset of an integrable polygon. In particular,
spectral degeneracy alone cannot produce such an asymptotically
avoiding sequence, since the lower bound remains valid uniformly over
every vector in each degenerate eigenspace.

We emphasize, however, that this result is a uniform
\emph{non-localization} statement rather than an equidistribution
theorem. The positivity of $C_2(V;\Omega)$ does not imply that
$|u_j|^2\,dx$ converges to normalized Lebesgue measure as
$\lambda_j\to\infty$, nor does it prescribe the limiting spatial
distribution of high-frequency eigenfunctions. It asserts instead
the weaker, but spectrum-wide, property that no fixed measurable
region of positive measure can become asymptotically invisible to a
sequence of normalized eigenfunctions.

\section{Extension to Schr\"odinger Operators with Bounded Potentials}
\label{sec:bounded-potentials}

The preceding results concern the free Dirichlet Laplacian. We now
show that the same non-localization mechanism is stable under the
addition of an arbitrary bounded real-valued potential for the
rectangular branch of the integrable geometries. More precisely, we
consider the Dirichlet Schr\"odinger operator
\[
    H_{\Omega,q}
    :=
    -\Delta_\Omega+q,
    \qquad
    q\in L^\infty(\Omega;\mathbb{R}),
\]
on rectangles and isosceles right triangles.
Since $q$ is real-valued and bounded,
$H_{\Omega,q}$ is self-adjoint on
\[
    D(H_{\Omega,q})=D(-\Delta_\Omega),
\]
as a bounded self-adjoint perturbation of the Dirichlet Laplacian.

The bounded-potential case requires some care under reflection.
If a Dirichlet eigenfunction is extended oddly across a side, the
potential must be extended evenly in order to preserve the
Schr\"odinger equation. Iterating this construction for a rectangle
leads to a problem on a rectangular flat torus, where the stationary
estimate of Bourgain, Burq, and Zworski
\cite[Theorem~1]{BourgainBurqZworski2013} can be applied. We first
make the reflection argument precise.

\subsection{Odd Reflection with Even Potentials}

\begin{lemma}[Odd reflection for Schr\"odinger eigenfunctions]
\label{lem:odd-reflection-potential}
Let $D\subset\mathbb{R}^2$ be a bounded polygonal domain and let
$S\subset\partial D$ be a straight side. Let $\rho$ denote reflection
across the line containing $S$, assume that $D$ and $\rho(D)$ have
disjoint interiors, and set
\[
    D^\ast
    :=
    \operatorname{int}
    \bigl(
        \overline{D}\cup\overline{\rho(D)}
    \bigr).
\]
Let
\[
    q\in L^\infty(D;\mathbb{R}),
\]
and define its even extension to $D^\ast$ by
\[
    q^\ast(x)
    :=
    \begin{cases}
        q(x),
        & x\in D,\\[1mm]
        q(\rho x),
        & x\in\rho(D).
    \end{cases}
\]
Suppose that $u\in H_0^1(D)$ satisfies
\begin{equation}
    \int_D
        \nabla u\cdot\nabla\overline{\varphi}\,dx
    +
    \int_D
        q\,u\,\overline{\varphi}\,dx
    =
    \lambda
    \int_D
        u\,\overline{\varphi}\,dx
    \label{eq:potential-weak-eigenfunction}
\end{equation}
for every $\varphi\in H_0^1(D)$. Define the odd extension
\[
    \widetilde u(x)
    :=
    \begin{cases}
        u(x),
        & x\in D,\\[1mm]
        -u(\rho x),
        & x\in\rho(D).
    \end{cases}
\]
Then
\[
    \widetilde u\in H_0^1(D^\ast)
\]
and
\[
    (-\Delta_{D^\ast}+q^\ast)\widetilde u
    =
    \lambda\widetilde u
\]
in the weak sense on $D^\ast$. Moreover,
\begin{equation}
    \|\widetilde u\|_{L^2(D^\ast)}^2
    =
    2\|u\|_{L^2(D)}^2.
    \label{eq:potential-reflection-norm}
\end{equation}
\end{lemma}

\begin{proof}
Since $u\in H_0^1(D)$ has zero trace on the reflecting side $S$,
the two reflected pieces have compatible traces across $S$.
As in Lemma~\ref{lem:odd-reflection}, the Sobolev gluing argument gives
\[
    \widetilde u\in H_0^1(D^\ast).
\]
Let $\psi\in H_0^1(D^\ast)$ and define, on $D$,
\[
    \varphi(x)
    :=
    \psi(x)-\psi(\rho x).
\]
Then, $\varphi\in H_0^1(D)$. Splitting the integrals over $D^\ast$
into the two reflected copies and making the change of variables
$x=\rho y$ on $\rho(D)$, we obtain, using the orthogonality of the
linear part of $\rho$ and the evenness of $q^\ast$,
\begin{align*}
&\int_{D^\ast}
    \nabla\widetilde u\cdot\nabla\overline{\psi}\,dx
+
\int_{D^\ast}
    q^\ast\widetilde u\,\overline{\psi}\,dx
\\
&\qquad =
\int_D
    \nabla u\cdot
    \nabla\overline{\bigl(\psi-\psi\circ\rho\bigr)}\,dx
+
\int_D
    q\,u\,
    \overline{\bigl(\psi-\psi\circ\rho\bigr)}\,dx.
\end{align*}
Applying Eq. \eqref{eq:potential-weak-eigenfunction} with
$\varphi=\psi-\psi\circ\rho$ gives
\begin{align*}
&\int_{D^\ast}
    \nabla\widetilde u\cdot\nabla\overline{\psi}\,dx
+
\int_{D^\ast}
    q^\ast\widetilde u\,\overline{\psi}\,dx
\\
&\qquad =
\lambda
\int_D
    u\,\overline{\bigl(\psi-\psi\circ\rho\bigr)}\,dx
=
\lambda
\int_{D^\ast}
    \widetilde u\,\overline{\psi}\,dx.
\end{align*}
Thus $\widetilde u$ satisfies
\[
    (-\Delta_{D^\ast}+q^\ast)\widetilde u
    =
    \lambda\widetilde u
\]
weakly on $D^\ast$.

Finally, reflection preserves Lebesgue measure and
$|\widetilde u(\rho x)|=|u(x)|$. Hence
\[
    \|\widetilde u\|_{L^2(D^\ast)}^2
    =
    2\|u\|_{L^2(D)}^2,
\]
which proves Eq. \eqref{eq:potential-reflection-norm}.
\end{proof}

\subsection{Uniform Non-Localization with Bounded Potentials}

The uniform non-localization result for the free Dirichlet Laplacian
extends to Dirichlet Schr\"odinger operators with bounded real-valued
potentials on the rectangular branch.
\begin{theorem}[Uniform non-localization with bounded potentials]
\label{thm:bounded-potential}
Let $\Omega\subset\mathbb{R}^2$ be either a rectangle or an
isosceles right triangle, and let
\[
    q\in L^\infty(\Omega;\mathbb{R}).
\]
Let
\[
    H_{\Omega,q}
    :=
    -\Delta_\Omega+q
\]
be the Dirichlet Schr\"odinger operator on $\Omega$. Then, for every
nonempty open set $\omega\subset\Omega$, there exists a constant
$c_{\Omega,\omega,q}>0$ such that
\begin{equation}
    \|u\|_{L^2(\omega)}
    \geq
    c_{\Omega,\omega,q}
    \|u\|_{L^2(\Omega)}
    \label{eq:bounded-potential-main}
\end{equation}
for every nonzero eigenfunction
\[
    H_{\Omega,q}u=\lambda u.
\]
The constant $c_{\Omega,\omega,q}$ is independent of $\lambda$, its
multiplicity, and the choice of eigenfunction within the corresponding
eigenspace.
\end{theorem}

\begin{proof}
We treat the two geometries separately.

\medskip
\noindent
\textbf{Rectangle.}
Let
\[
    \Omega=R=(0,a)\times(0,b).
\]
Suppose that
\[
    (-\Delta_R+q)u=\lambda u,
    \qquad
    u\in H_0^1(R)\setminus\{0\}.
\]
Now, we extend $u$ oddly across the sides of $R$ and then periodically with
periods $2a$ and $2b$. Let us denote the resulting function on the
rectangular flat torus
\[
    \mathbb{T}_{a,b}^2
    :=
    \mathbb{R}^2/
    \bigl(
        2a\mathbb{Z}\times2b\mathbb{Z}
    \bigr)
\]
by $\widetilde u$. 
Next, we extend $q$ evenly across the sides of $R$ and then
periodically with the same periods, and denote the resulting potential
by $\widetilde q$. Then
\[
    \widetilde q
    \in
    L^\infty(\mathbb{T}_{a,b}^2;\mathbb{R})
    \subset
    L^2(\mathbb{T}_{a,b}^2;\mathbb{R}).
\]
Applying Lemma~\ref{lem:odd-reflection-potential} successively gives
\begin{equation}
    (-\Delta_{\mathbb{T}_{a,b}^2}+\widetilde q)\widetilde u
    =
    \lambda\widetilde u
    \label{eq:potential-torus-eigenfunction}
\end{equation}
in the weak sense on $\mathbb{T}_{a,b}^2$.

A fundamental domain
$(-a,a)\times(-b,b)$ consists of four reflected copies of $R$.
Consequently,
\begin{equation}
    \|\widetilde u\|_{L^2(\mathbb{T}_{a,b}^2)}^2
    =
    4\|u\|_{L^2(R)}^2.
    \label{eq:potential-rectangle-norm}
\end{equation}
Moreover, on the original copy of $R$ one has $\widetilde u=u$.
Hence
\begin{equation}
    \|\widetilde u\|_{L^2(\omega)}
    =
    \|u\|_{L^2(\omega)}.
    \label{eq:potential-rectangle-local}
\end{equation}

The stationary estimate of Bourgain, Burq, and Zworski
\cite[Theorem~1]{BourgainBurqZworski2013}, applied to the
rectangular torus $\mathbb{T}_{a,b}^2$, the real-valued potential
$\widetilde q$, and the nonempty open set $\omega$, gives a constant
$K_{\mathbb{T}_{a,b}^2,\omega,\widetilde q}>0$, independent of the
spectral parameter, such that
\begin{equation}
    \|v\|_{L^2(\mathbb{T}_{a,b}^2)}
    \leq
    K_{\mathbb{T}_{a,b}^2,\omega,\widetilde q}
    \left(
        \|(-\Delta+\widetilde q-\lambda)v\|_
            {L^2(\mathbb{T}_{a,b}^2)}
        +
        \|v\|_{L^2(\omega)}
    \right).
    \label{eq:BBZ-stationary}
\end{equation}
Taking $v=\widetilde u$ and using
Eq. \eqref{eq:potential-torus-eigenfunction}, we obtain
\[
    \|\widetilde u\|_{L^2(\omega)}
    \geq
    K_{\mathbb{T}_{a,b}^2,\omega,\widetilde q}^{-1}
    \|\widetilde u\|_{L^2(\mathbb{T}_{a,b}^2)}.
\]
Combining this with
Eq. \eqref{eq:potential-rectangle-norm} and
Eq. \eqref{eq:potential-rectangle-local} yields
\begin{equation}
    \|u\|_{L^2(\omega)}
    \geq
    \frac{2}
    {K_{\mathbb{T}_{a,b}^2,\omega,\widetilde q}}
    \|u\|_{L^2(R)}.
    \label{eq:potential-rectangle-final}
\end{equation}
Thus, the inequality \eqref{eq:bounded-potential-main} holds for rectangles.

\medskip
\noindent
\textbf{Isosceles right triangle.}
Let $\Omega=T_{\mathrm{iso}}$ be an isosceles right triangle and let
$Q$ be the square obtained by reflecting $T_{\mathrm{iso}}$ across
its hypotenuse. Let $\rho$ denote this reflection.

Now, we define the even extension of $q$ to $Q$ by
\[
    q_Q(x)
    :=
    \begin{cases}
        q(x),
        & x\in T_{\mathrm{iso}},\\[1mm]
        q(\rho x),
        & x\in\rho(T_{\mathrm{iso}}).
    \end{cases}
\]
Then,
\[
    q_Q\in L^\infty(Q;\mathbb{R}).
\]
Let $\widetilde u$ denote the odd reflection of $u$ across the
hypotenuse. By Lemma~\ref{lem:odd-reflection-potential},
\[
    (-\Delta_Q+q_Q)\widetilde u
    =
    \lambda\widetilde u
\]
and
\begin{equation}
    \|\widetilde u\|_{L^2(Q)}^2
    =
    2\|u\|_{L^2(T_{\mathrm{iso}})}^2.
    \label{eq:potential-isosceles-norm}
\end{equation}
Since $\omega\subset T_{\mathrm{iso}}$ and
$\widetilde u=u$ on the original triangle,
\[
    \|\widetilde u\|_{L^2(\omega)}
    =
    \|u\|_{L^2(\omega)}.
\]
Applying the rectangle case to the square $Q$, with potential $q_Q$
and observation set $\omega$, gives a constant
$c_{Q,\omega,q_Q}>0$ such that
\[
    \|\widetilde u\|_{L^2(\omega)}
    \geq
    c_{Q,\omega,q_Q}
    \|\widetilde u\|_{L^2(Q)}.
\]
Using Eq. \eqref{eq:potential-isosceles-norm}, we conclude that
\begin{equation}
    \|u\|_{L^2(\omega)}
    \geq
    \sqrt{2}\,
    c_{Q,\omega,q_Q}
    \|u\|_{L^2(T_{\mathrm{iso}})}.
    \label{eq:potential-isosceles-final}
\end{equation}
This proves the theorem.
\end{proof}

\begin{remark}
\label{rem:bounded-potential-scope}
Theorem~\ref{thm:bounded-potential} is stated for rectangles and
isosceles right triangles because the reflection procedure reduces
these domains to rectangular flat tori, where the required stationary
estimate with a rough potential is available. The argument does not,
by itself, provide the corresponding result for equilateral or
hemi-equilateral triangles. Such an extension would require an
appropriate observability or stationary estimate for Schr\"odinger
operators with nontrivial potentials in the equilateral-triangle
geometry.
\end{remark}

\begin{remark}
\label{rem:bounded-potential-observation-set}
Unlike the free-Laplacian results of
Sections~\ref{sec:uniform-L2-nonlocalization} and
\ref{sec:integrable-polygons}, which allow arbitrary measurable
observation sets of positive Lebesgue measure,
Theorem~\ref{thm:bounded-potential} is stated for nonempty open
observation sets. This distinction reflects the observation class
in the stationary estimate used in its proof.
\end{remark}

\section{Robust Non-Localization for Quasimodes and Spectral Clusters}
\label{sec:quasimodes}

The preceding results concern exact eigenfunctions. We now consider
approximate eigenfunctions and narrow spectral clusters. The connection
between observability and estimates involving the spectral defect
$(A-\lambda)u$ is classical in control theory and is commonly formulated
through resolvent or Hautus-type inequalities; see, for example,
\cite{Miller2012} and the references therein.

Here, we use the time-dependent observability estimate directly to
derive an explicit defect bound adapted to the present
non-localization problem. The resulting estimate keeps track of the
observation time and observability constant and applies without
reference to the particular polygonal geometry. Exact eigenfunctions
correspond to zero spectral defect, while sufficiently accurate
quasimodes and vectors in narrow spectral clusters inherit a positive
lower bound on the observation set.

\subsection{An Explicit Defect Estimate}
\begin{proposition}[Explicit defect estimate from observability]
\label{prop:defect-estimate}
Let $A$ be a self-adjoint operator on $L^2(\Omega)$. Assume that, for
some measurable set $V\subset\Omega$, some observation time
$\tau>0$, and some constant $C_{V,\tau}>0$, the observability estimate
\begin{equation}
    \|f\|_{L^2(\Omega)}^2
    \leq
    C_{V,\tau}
    \int_0^\tau
    \|e^{-itA}f\|_{L^2(V)}^2
    \,dt
    \label{eq:quasimode-observability}
\end{equation}
holds for every $f\in L^2(\Omega)$.

Then, for every $\lambda\in\mathbb{R}$ and every
$u\in\mathcal{D}(A)$, one has
\begin{equation}
{
    \|u\|_{L^2(V)}
    \geq
    \frac{1}{\sqrt{\tau C_{V,\tau}}}
    \|u\|_{L^2(\Omega)}
    -
    \frac{\tau}{\sqrt{3}}
    \|(A-\lambda)u\|_{L^2(\Omega)}.
    }
    \label{eq:quasimode-main}
\end{equation}
\end{proposition}

\begin{proof}
Fix $\lambda\in\mathbb{R}$ and $u\in\mathcal{D}(A)$, and set
\[
    r:=(A-\lambda)u.
\]
Duhamel's formula gives
\begin{equation}
    e^{-itA}u
    =
    e^{-it\lambda}u
    -
    i
    \int_0^t
    e^{-i(t-s)A}
    e^{-is\lambda}r
    \,ds.
    \label{eq:quasimode-duhamel}
\end{equation}
Since $A$ is self-adjoint, the propagator $e^{-itA}$ is unitary on
$L^2(\Omega)$. Hence
\begin{align}
    \left\|
        e^{-itA}u-e^{-it\lambda}u
    \right\|_{L^2(\Omega)}
    &\leq
    \int_0^t
    \|r\|_{L^2(\Omega)}
    \,ds
    \notag\\
    &=
    t\|r\|_{L^2(\Omega)}.
    \label{eq:quasimode-pointwise-error}
\end{align}
Restricting to $V$ gives
\[
    \left\|
        e^{-itA}u-e^{-it\lambda}u
    \right\|_{L^2(V)}
    \leq
    t\|r\|_{L^2(\Omega)}.
\]
By the triangle inequality in
$L^2((0,\tau);L^2(V))$,
\begin{align}
&
\left(
    \int_0^\tau
    \|e^{-itA}u\|_{L^2(V)}^2
    \,dt
\right)^{1/2}
\notag\\
&\qquad\leq
\left(
    \int_0^\tau
    \|e^{-it\lambda}u\|_{L^2(V)}^2
    \,dt
\right)^{1/2}
+
\left(
    \int_0^\tau
    t^2
    \|r\|_{L^2(\Omega)}^2
    \,dt
\right)^{1/2}
\notag\\
&\qquad=
\sqrt{\tau}\,
\|u\|_{L^2(V)}
+
\frac{\tau^{3/2}}{\sqrt{3}}\,
\|r\|_{L^2(\Omega)}.
\label{eq:quasimode-time-bound}
\end{align}
On the other hand, applying
the inequality \eqref{eq:quasimode-observability} to $u$ yields
\[
    \frac{1}{\sqrt{C_{V,\tau}}}
    \|u\|_{L^2(\Omega)}
    \leq
    \left(
        \int_0^\tau
        \|e^{-itA}u\|_{L^2(V)}^2
        \,dt
    \right)^{1/2}.
\]
Combining the preceding two estimates gives
\[
    \frac{1}{\sqrt{C_{V,\tau}}}
    \|u\|_{L^2(\Omega)}
    \leq
    \sqrt{\tau}\,
    \|u\|_{L^2(V)}
    +
    \frac{\tau^{3/2}}{\sqrt{3}}
    \|(A-\lambda)u\|_{L^2(\Omega)}.
\]
Dividing by $\sqrt{\tau}$ yields the inequality
\eqref{eq:quasimode-main}.
\end{proof}

Proposition~\ref{prop:defect-estimate} gives a quantitative
stability estimate in terms of the spectral defect
$\|(A-\lambda)u\|_{L^2(\Omega)}$. Exact eigenfunctions correspond to
zero defect and recover the stationary consequence of observability
used earlier in the paper.

\subsection{Uniform Non-Localization of Quasimodes}

\begin{corollary}[Uniform non-localization of accurate quasimodes]
\label{cor:quasimode-nonlocalization}
Under the assumptions of
Proposition~\ref{prop:defect-estimate}, suppose that
$u\in\mathcal{D}(A)\setminus\{0\}$ satisfies
\begin{equation}
    \|(A-\lambda)u\|_{L^2(\Omega)}
    \leq
    \varepsilon
    \|u\|_{L^2(\Omega)}
    \label{eq:quasimode-defect}
\end{equation}
for some $\varepsilon\geq0$. Then,
\begin{equation}
    \|u\|_{L^2(V)}
    \geq
    \left(
        \frac{1}{\sqrt{\tau C_{V,\tau}}}
        -
        \frac{\tau\varepsilon}{\sqrt{3}}
    \right)
    \|u\|_{L^2(\Omega)}.
    \label{eq:quasimode-corollary}
\end{equation}
In particular, if
\begin{equation}
    \varepsilon
    <
    \frac{\sqrt{3}}
         {\tau\sqrt{\tau C_{V,\tau}}},
    \label{eq:quasimode-threshold}
\end{equation}
then the coefficient on the right-hand side of the inequality
\eqref{eq:quasimode-corollary} is strictly positive.
\end{corollary}

\begin{proof}
Substituting
the inequality \eqref{eq:quasimode-defect} into
the inequality \eqref{eq:quasimode-main} gives the inequality
\eqref{eq:quasimode-corollary}. The final assertion follows from
the positivity condition
\[
    \frac{1}{\sqrt{\tau C_{V,\tau}}}
    -
    \frac{\tau\varepsilon}{\sqrt{3}}
    >0.
\]
\end{proof}

When $\varepsilon=0$, this reduces to the estimate of the eigenfunction case:
\[
    \|u\|_{L^2(V)}
    \geq
    \frac{1}{\sqrt{\tau C_{V,\tau}}}
    \|u\|_{L^2(\Omega)}.
\]
Hence, the same lower bound remains effective for approximate
eigenfunctions as long as the spectral defect is sufficiently small.

\subsection{Narrow Spectral Clusters}

The robust estimate also yields uniform control over vectors formed
from several distinct eigenvalues contained in a sufficiently narrow
spectral window.
For $\lambda\in\mathbb{R}$ and $\delta\geq0$, we define
\begin{equation}
    \mathcal{E}_{\lambda,\delta}(A)
    :=
    \operatorname{Ran}
    \mathbf{1}_{[\lambda-\delta,\lambda+\delta]}(A),
    \label{eq:spectral-subspace}
\end{equation}
where $\mathbf{1}_{[\lambda-\delta,\lambda+\delta]}(A)$ denotes the
spectral projection of the self-adjoint operator $A$ associated with
the interval $[\lambda-\delta,\lambda+\delta]$.
When $A$ has discrete spectrum, with
\[
    A\phi_j=\lambda_j\phi_j,
\]
this space can equivalently be written as
\[
    \mathcal{E}_{\lambda,\delta}(A)
    =
    \operatorname{span}
    \left\{
        \phi_j:
        |\lambda_j-\lambda|\leq\delta
    \right\}.
\]
Thus, $\mathcal{E}_{\lambda,\delta}(A)$ consists of arbitrary linear
combinations of eigenfunctions whose eigenvalues lie in the spectral
window $[\lambda-\delta,\lambda+\delta]$.

\begin{corollary}[Uniform non-localization on narrow spectral clusters]
\label{cor:spectral-cluster-nonlocalization}
Assume, in addition to the hypotheses of
Proposition~\ref{prop:defect-estimate}, that $A$ has compact
resolvent. 
Then, for every $\lambda\in\mathbb{R}$, every
$\delta\geq0$, and every $u\in\mathcal{E}_{\lambda,\delta}(A)$,
one has
\begin{equation}
{
    \|u\|_{L^2(V)}
    \geq
    \left(
        \frac{1}{\sqrt{\tau C_{V,\tau}}}
        -
        \frac{\tau\delta}{\sqrt{3}}
    \right)
    \|u\|_{L^2(\Omega)}.
    }
    \label{eq:spectral-cluster-main}
\end{equation}
Consequently, if
\begin{equation}
    \delta
    <
    \frac{\sqrt{3}}
         {\tau\sqrt{\tau C_{V,\tau}}},
    \label{eq:spectral-cluster-threshold}
\end{equation}
then
\[
    c_{V,\tau,\delta}
    :=
    \frac{1}{\sqrt{\tau C_{V,\tau}}}
    -
    \frac{\tau\delta}{\sqrt{3}}
    >0,
\]
and
\[
    \|u\|_{L^2(V)}
    \geq
    c_{V,\tau,\delta}
    \|u\|_{L^2(\Omega)}
\]
uniformly in the center $\lambda$ of the spectral window and in the
choice of
$u\in\mathcal{E}_{\lambda,\delta}(A)$.
\end{corollary}

\begin{proof}
By the spectral theorem, if
$u\in\mathcal{E}_{\lambda,\delta}(A)$, then the spectral support of
$u$ is contained in $[\lambda-\delta,\lambda+\delta]$. Therefore,
\[
    \|(A-\lambda)u\|_{L^2(\Omega)}
    \leq
    \delta\|u\|_{L^2(\Omega)}.
\]
The conclusion follows from
Proposition~\ref{prop:defect-estimate}.

\end{proof}

\begin{remark}
\label{rem:cluster-vs-eigenspace}
When $\delta=0$,
Corollary~\ref{cor:spectral-cluster-nonlocalization} reduces to the
eigenspace-uniform non-localization estimate. For $\delta>0$, the
statement also applies to arbitrary linear combinations of
eigenfunctions associated with distinct eigenvalues, provided that
these eigenvalues lie in the same spectral window. Thus, the
conclusion is stable not only under spectral multiplicity but also
under linear combinations across sufficiently narrow spectral
windows.
\end{remark}

\subsection{Application to Integrable Polygons}

We finally specialize the abstract estimate to the Dirichlet
Laplacian on the integrable polygonal domains considered above.

\begin{corollary}[Spectral-cluster non-localization on integrable polygons]
\label{cor:polygon-spectral-cluster}
Let $\Omega\subset\mathbb{R}^2$ be an integrable polygon and let
$V\subset\Omega$ be measurable with $|V|>0$. Let
\[
    A=-\Delta_\Omega
\]
be the self-adjoint Dirichlet Laplacian and fix an observation time
$\tau>0$ for which the Schr\"odinger observability estimate
\begin{equation}
    \|f\|_{L^2(\Omega)}^2
    \leq
    C_{\Omega,V,\tau}
    \int_0^\tau
    \left\|
        e^{it\Delta_\Omega}f
    \right\|_{L^2(V)}^2
    \,dt
    \label{eq:polygon-cluster-observability}
\end{equation}
holds.

Then, for every $\lambda\in\mathbb{R}$, every $\delta\geq0$, and every
\[
    u\in
    \operatorname{Ran}
    \mathbf{1}_{[\lambda-\delta,\lambda+\delta]}
    (-\Delta_\Omega),
\]
one has
\begin{equation}
    \|u\|_{L^2(V)}
    \geq
    \left(
        \frac{1}{\sqrt{\tau C_{\Omega,V,\tau}}}
        -
        \frac{\tau\delta}{\sqrt{3}}
    \right)
    \|u\|_{L^2(\Omega)}.
    \label{eq:polygon-cluster-bound}
\end{equation}
In particular, if
\begin{equation}
    \delta
    <
    \frac{\sqrt{3}}
         {\tau\sqrt{\tau C_{\Omega,V,\tau}}},
    \label{eq:polygon-cluster-threshold}
\end{equation}
then the lower bound is strictly positive and uniform both in the
center $\lambda$ of the spectral window and in the choice of vector
within the corresponding spectral subspace.
\end{corollary}

\begin{proof}
The Dirichlet Laplacian on a bounded polygonal domain has compact
resolvent. Moreover,
\[
    e^{-itA}
    =
    e^{it\Delta_\Omega}.
\]
The required Schr\"odinger observability estimate is available for
the integrable polygonal domains considered in the preceding
sections. Hence
Corollary~\ref{cor:spectral-cluster-nonlocalization} applies directly
and yields the inequality \eqref{eq:polygon-cluster-bound}.
\end{proof}

\begin{remark}
\label{rem:cluster-width}
The condition
\eqref{eq:spectral-cluster-threshold} is a sufficient condition for
the property of uniform non-localization on the corresponding spectral cluster and
is not claimed to be optimal. In particular, the factor
$\tau/\sqrt{3}$ in Proposition~\ref{prop:defect-estimate} comes
directly from the Duhamel representation used in its proof and from
\[
    \int_0^\tau t^2\,dt
    =
    \frac{\tau^3}{3}.
\]
A sharper estimate of the dynamical error term could improve the
resulting sufficient condition on the spectral-window width.
\end{remark}

\section{Scope, Limitations, and Further Work}
\label{sec:scope-limitations}

The results of this paper have three different levels of scope.

First, for the free Dirichlet Laplacian, the uniform non-localization
result applies to the four integrable polygonal classes considered
here: rectangles, isosceles right triangles, equilateral triangles,
and hemi-equilateral triangles. Their reflection and tiling structures
reduce the problem to geometries for which the required
Schr\"odinger observability estimates are available
\cite{AlphonseLafontaine2025,BurqZworski2019,SpectralInvariants2025,McCartin2008}. The same argument does not presently extend
to arbitrary rational polygons.

For a general rational polygon, unfolding produces a compact
translation surface that may have conical singularities. The
observability estimates used above are not presently available in
this setting, so the same argument does not directly extend to
arbitrary rational polygons.

To the best of our knowledge, the corresponding uniform
non-localization problem remains open for regular $N$-gons with
$N\geq5$. We formulate this question explicitly.
\begin{problem}[Uniform non-localization on regular polygons]
\label{prob:regular-polygons}
Let $P_N\subset\mathbb{R}^2$ be a regular $N$-gon with $N\geq5$.
Is it true that, for every measurable set $V\subset P_N$ with
$|V|>0$,
\[
    \inf_{\lambda\in\sigma(-\Delta_{P_N})}
    \inf_{\substack{
        u\in E_\lambda(P_N)\setminus\{0\}
    }}
    \frac{\|u\|_{L^2(V)}}
         {\|u\|_{L^2(P_N)}}
    >0?
\]
\end{problem}

More generally, the same question is natural for arbitrary rational
polygons. From the viewpoint developed here, one possible route
toward such an extension is to establish Schr\"odinger observability
on the corresponding unfolded translation surfaces. Indeed, an
estimate of the form
\[
    \|f\|_{L^2(\Omega)}^2
    \leq
    C_{V,\tau}
    \int_0^\tau
    \|e^{it\Delta_\Omega}f\|_{L^2(V)}^2
    \,dt
\]
would immediately imply the corresponding uniform
non-localization bound on each eigenspace via
Lemma~\ref{lem:obs-to-eigen}.

Second, the bounded-potential extension obtained in
Section~\ref{sec:bounded-potentials} is presently restricted to
rectangles and isosceles right triangles. This restriction comes from
the stationary estimates available for Schr\"odinger operators with
rough potentials on rectangular flat tori. The reflection argument
itself extends the Dirichlet eigenfunction oddly and the potential
evenly across the relevant sides. To treat equilateral and
hemi-equilateral triangles in the same way, one would need a
corresponding stationary or observability estimate for
\[
    -\Delta+q
\]
in the equilateral-triangle setting. Such an estimate would extend
the bounded-potential result to the equilateral and
hemi-equilateral cases.

There is also a difference in the admissible observation sets. For
the free Laplacian, the observability results used here apply to
measurable sets of positive Lebesgue measure, whereas the stationary
rough-potential estimate used in
Theorem~\ref{thm:bounded-potential}
is available here for nonempty open observation sets. It would be of
interest to determine whether the bounded-potential result continues
to hold for arbitrary measurable sets of positive measure.

Third, the quasimode and spectral-cluster estimates of
Section~\ref{sec:quasimodes} are quantitative but not claimed to be
optimal. The cluster-width condition
\[
    \delta
    <
    \frac{\sqrt{3}}
         {\tau\sqrt{\tau C_{V,\tau}}}
\]
is a sufficient condition obtained from the elementary Duhamel
estimate in
Proposition~\ref{prop:defect-estimate}. No sharpness is asserted
for either the factor $\tau/\sqrt{3}$ or the resulting admissible
spectral window. Improved control of the dynamical error term, or
additional information on the observability constant
$C_{V,\tau}$, could lead to larger spectral windows and sharper
quasimode estimates.

The defect estimate is not specific to the polygonal setting. Whenever
the Schr\"odinger flow of a self-adjoint operator satisfies a uniform
observability inequality, Proposition~\ref{prop:defect-estimate}
provides a lower mass bound for approximate eigenfunctions. This
raises similar questions for other boundary conditions, mixed
boundary problems, and Schr\"odinger operators on polygonal or more
general domains where suitable observability estimates are available.

Further extensions may proceed in two directions. One is to establish
the required observability estimates for a broader class of operators
and geometries. The other is to improve the quantitative dependence
of the resulting lower bounds on the observation set, the potential,
and the spectral defect.

\section{Conclusion}
\label{sec:conclusion}

We have established a spectrum-wide $L^2$ non-localization principle
for the complete class of integrable polygons in the trigonometric
spectral classification: rectangles, isosceles right triangles,
equilateral triangles, and hemi-equilateral triangles. For every such
domain $\Omega$ and every measurable set $V\subset\Omega$ with
$|V|>0$, we prove
\[
    C_2(V;\Omega)>0.
\]
Equivalently, every positive-measure observation region captures a
uniformly positive fraction of the $L^2$ mass of every Dirichlet
Laplacian eigenfunction. The lower bound is uniform over the spectrum
and over the complete eigenspace associated with each eigenvalue.
Thus spectral degeneracy, including arbitrary cancellations within a
multiple eigenspace, does not permit asymptotic avoidance of a fixed
observation set.

For rectangles, we obtain a substantially more quantitative
description. Let
\[
    R=(0,a)\times(0,b),
    \qquad
    \alpha=\frac{|V|}{|R|},
\]
and assume that the reflected extension of $V$ to the associated
rectangular torus has finite perimeter. We derive an explicit
sufficient spectral threshold $\lambda_*(R,V)$, determined by the
relative area and the reflected perimeter of the observation set,
such that every Dirichlet eigenfunction with
\[
    \lambda\geq\lambda_*(R,V)
\]
satisfies the area-only estimate
\[
    \frac{\|u\|_{L^2(V)}}{\|u\|_{L^2(R)}}
    \geq
    \left[
        \frac{\alpha}{2}
        \left(
            1-
            \frac{\sin(\pi\alpha)}{\pi\alpha}
        \right)
    \right]^{1/2}.
\]
Thus the geometry of the observation set determines the onset of the
high-frequency regime, whereas the resulting mass floor depends only
on the relative observed area. A finite-bandwidth
Tur\'an--Nazarov argument controls the complementary low-frequency
regime and yields a quantitative full-spectrum estimate
\[
    C_2(V;R)
    \geq
    F_R(V)>0,
\]
where the dependence on the low-frequency regime is explicit up to
the universal Tur\'an--Nazarov constant.

The quantitative rectangle argument separates the mechanisms
responsible for these two regimes. At high frequency, an extremal
Fourier estimate for normalized indicator measures converts the area
fraction $\alpha$ into the gap
\[
    \eta(\alpha)
    =
    1-
    \frac{\sin(\pi\alpha)}{\pi\alpha}.
\]
Finite perimeter provides quantitative decay of the remaining Fourier
coefficients, while the arithmetic geometry of lattice points on the
dual-lattice frequency circle gives an explicit two-point clustering
scale. Together these ingredients produce a controlled inter-cluster
error and, ultimately, the explicit onset threshold
$\lambda_*(R,V)$. In particular, the argument converts the
Fourier-analytic quantities arising in stationary toral observability
into geometric parameters of the observation set.

At the qualitative level, a central feature of the analysis is its
eigenspace-level formulation. Rather than estimating a preferred
basis of separated modes, the argument treats complete spectral
subspaces and is therefore insensitive to spectral multiplicity.
Combined with the reflection and symmetry structures of the
integrable polygonal geometries, this provides a unified mechanism
for passing from dynamical observability to stationary
non-localization.

The same viewpoint also gives two further forms of stability. First,
on rectangles and isosceles right triangles, the stationary estimate
persists for Dirichlet Schr\"odinger operators
\[
    -\Delta_\Omega+q,
    \qquad
    q\in L^\infty(\Omega;\mathbb R),
\]
through compatible reflection of both the eigenfunction and the
potential. The resulting estimate remains uniform in the eigenvalue,
its multiplicity, and the choice of vector within the corresponding
eigenspace.

Second, the stationary estimate is stable under a controlled spectral
defect. If $A$ is self-adjoint and its Schr\"odinger flow satisfies an
observability inequality on $V$ over a time interval of length $\tau$,
then
\[
    \|u\|_{L^2(V)}
    \geq
    \frac{1}{\sqrt{\tau C_{V,\tau}}}
    \|u\|_{L^2(\Omega)}
    -
    \frac{\tau}{\sqrt{3}}
    \|(A-\lambda)u\|_{L^2(\Omega)}.
\]
Exact eigenfunctions correspond to zero defect, while sufficiently
accurate quasimodes retain a positive observation bound. Through the
spectral theorem, the same estimate gives uniform lower bounds for
arbitrary vectors in sufficiently narrow spectral windows.

The results therefore reveal several complementary levels of
non-localization. On the complete class of integrable polygons,
non-localization is stable under spectral degeneracy. On the
rectangular branch it admits a quantitative full-spectrum refinement,
with an area-only high-frequency mass floor and a geometry-dependent
onset threshold. The same stationary framework further persists under
bounded potentials on the rectangular branch and under sufficiently
small spectral defects.

We emphasize that these conclusions concern uniform
non-localization, not equidistribution. They exclude asymptotic
avoidance of a fixed positive-measure observation set but do not
determine the limiting spatial distribution of high-frequency
eigenfunctions. Several quantitative questions remain open. In
particular, the explicit threshold obtained here is sufficient rather
than optimal, and the sharp dependence of the high-frequency
non-localization constant on the small-area parameter remains to be
determined. Beyond the integrable setting, spectrum-wide
non-localization for general rational or nonintegrable polygons
remains open. Extending the present theory to such geometries would
require new observability or stationary estimates on the associated
unfolded spaces, while extending the bounded-potential result beyond
the rectangular branch requires comparable rough-potential estimates
in the corresponding polygonal geometries.

\bibliographystyle{elsarticle-num}
\bibliography{refs}

\end{document}